\documentclass[11pt, reqno]{amsart}
\usepackage{amsfonts, amssymb, amscd}
\usepackage{graphicx}
\usepackage{hyperref}
\usepackage{parskip}
\usepackage{slashed}
\usepackage{fullpage}

\numberwithin{equation}{section}

\newtheorem{thm}{Theorem}[section]

\newtheorem{lem}[thm]{Lemma}

\newtheorem{prop}[thm]{Proposition}

\theoremstyle{remark}
\newtheorem{rem}[thm]{Remark}

\theoremstyle{definition}

\renewcommand{\Re}{\mathrm{Re}}
\renewcommand{\Im}{\mathrm{Im}}

\newcommand{\R}{\mathbb{R}}
\newcommand{\Z}{\mathbb{Z}}
\newcommand{\C}{\mathbb{C}}

\newcommand{\Sp}{\mathbb{S}}
\renewcommand{\H}{\mathbb{H}}

\newcommand{\cK}{\mathcal{K}}
\newcommand{\cN}{\mathcal{N}}

\newcommand{\Ec}{E_{\mathrm{crit}}}

\newcommand{\Lts}{L^{3,6}_\theta}
\newcommand{\Lst}{L^{6,3}_\theta}

\DeclareMathOperator{\diag}{diag}
\DeclareMathOperator{\supp}{supp}

\title{A controlling norm for Energy-critical Schr\"odinger maps}

\author[B. Dodson]{Benjamin Dodson}
\address{University of California, Berkeley}
\email{benjadod@math.berkeley.edu}
\author[P. Smith]{Paul Smith}
\address{University of California, Berkeley}
\email{smith@math.berkeley.edu}

\thanks{The first author was supported by NSF grant DMS-1103914
and the second by NSF grant DMS-1103877.}

\begin{document}

\begin{abstract}
We consider energy-critical Schr\"odinger maps with target either the sphere $\Sp^2$
or hyperbolic plane $\H^2$ and establish that a unique solution may be continued so
long as a certain space-time $L^4$ norm remains bounded. This reduces the large data 
global wellposedness problem to that of controlling this norm.
\end{abstract}

\maketitle

\tableofcontents

\section{Introduction}

We consider the Schr\"odinger map equation on $\R^{2+1}$
with target either the sphere $\Sp^2$ or hyperbolic plane $\H^2$.
With the appropriate modeling we may interpret
$\Sp^2$ and $\H^2$ as submanifolds of $\R^3$, e.g.,
\[
\begin{split}
\Sp^2 &= \{ y = (y_0, y_1, y_2) \in \R^3: y_0^2 + y_1^2 + y_2^2 = 1 \} \\
\H^2 &= \{y = (y_0, y_1, y_2) \in \R^3: -y_0^2 - y_1^2 + y_2^2 = 1, y_2 \geq 0 \}
\end{split}
\]
with the Riemannian structures induced by the Euclidean metric
$dy_0^2 + dy_1^2 + dy_2^2$ in the case of $\Sp^2$ and by the
Minkowski metric $-dy_0^2 + dy_1^2 + dy_2^2$ in the case of $\H^2$.
Setting $\eta_\mu := \diag(1,1,\mu)$, we define the cross product
$\times_\mu$ by $v \times_\mu w := \eta_\mu \cdot (v \times w)$.
The Schr\"odinger map initial-value problem then takes the form
\begin{equation} \label{SM}
\partial_t \phi = \phi \times_\mu \Delta \phi ,
\quad \quad
\phi(0, x) = \phi_0(x)
\end{equation}
where $\phi$ is assumed to take values in $\Sp^2$ or $\H^2$ according
to whether $\mu = +1$ or $\mu = -1$, respectively.
Schr\"odinger maps admit the conserved energy
\begin{equation} \label{energy}
E(\phi) := \frac12 \int_{\R^2} \lvert \nabla \phi \rvert^2_\mu dx
\end{equation}
and both the equation \eqref{SM} and the energy \eqref{energy} are invariant 
with respect to scalings
\[
\phi(t, x) \mapsto \phi(\lambda^2 t, \lambda x)
\]
The problem we study is therefore \emph{energy critical}.

When the target is $\Sp^2$, Schr\"odinger maps arise as a 
Heisenberg model, i.e., a nearest-neighbor spin model, of ferromagnetism \cite{SuSuBa86}.
The Schr\"odinger map equation is also known as the Landau-Lifshitz equation. See
\cite{La67, SuSuBa86, PaTo91, MaPaSo94, ChShUh00, NaStUh03} and
the references therein.

For the local theory, see \cite{Mc07, SuSuBa86}.
\begin{thm}[Local existence] \label{thm:local}
If $\phi_0 \in \dot{H}^1 \cap \dot{H}^3$, then there exists a time $T > 0$
such that \eqref{SM} has a unique solution in $L^\infty_t([0, T] : \dot{H}^1 \cap \dot{H}^3)$.
\end{thm}

The small data problem with the sphere as target has been intensely studied,
see \cite{Be08a, Be08b, BeIoKe07, BeIoKeTa11, ChShUh00, IoKe06, IoKe07b}.
Global wellposedness for data with small energy is established in \cite{BeIoKeTa11}.
In \cite{Sm10} this result is extended to data small in the scale invariant Besov space $\dot{B}^1_{2, \infty}$,
assuming control on a certain space-time $L^4$ norm (the same that we assume in this article);
the result still holds even when the assumption of space-time $L^4$ control is dropped \cite{Sm11}.
Large data results have been established in some special settings.
Global wellposedness and scattering hold for radial Schr\"odinger maps into the sphere \cite{GuKo11},
though interestingly the problem is open for radial maps into $\H^2$.
Global wellposedness and scattering also hold for equivariant maps into the sphere
with energy less than $4\pi$, see \cite{BeIoKeTa11b}; the same hold true for
equivariant maps into $\H^2$ with finite energy \cite{BeIoKeTa12}.

Instead of working with \eqref{SM} directly, we work at the level of the differentiated system.
This leads to a system with coupled $L^2$-critical covariant Schr\"odinger equations
that exhibits gauge invariance. The details of this approach, first undertaken
in \cite{ChShUh00} in the context of Schr\"odinger map wellposedness problems, are presented in the next section.
The differentiated system cannot be wellposed unless the gauge freedom is eliminated
by making a gauge choice. We adopt the caloric gauge, which was first introduced in \cite{Tao04}
to study wave maps and then used for Schr\"odinger maps (with small energy) for the first time in \cite{BeIoKeTa11}.
Details of the construction for the target $\H^2$ are established in \cite{T4} and, 
for bounded geometry settings up to the ground state, in \cite{Sm09}.
We describe the caloric gauge construction in the next section and discuss
the bounds available in that gauge.

Loosely speaking, our main result is that a Schr\"odinger map $\phi$ may be continued
in time so long as the $\ell^2 L^4$ norm of $\nabla \phi$ is finite. The $\ell^2 L^4$ norm
for functions $f$ is defined by
\[
\lVert f \rVert_{\ell^2 L^4} := \left( \sum_{k \in \Z} \lVert P_k f \rVert_{L^4_{t,x}}^2 \right)^{\frac12}
\]
where $P_k$ denote standard Littlewood-Paley projections to dyadic frequency shells of
size $\sim 2^k$. From this definition it is clear that we have the embedding
$\ell^2 L^4 \hookrightarrow L^4$. For technical reasons, our theorem applies to maps
with energy less than the energy of the ground state, denoted by $\Ec$, which is
$4 \pi$ when the target is $\Sp^2$ and $+\infty$ when the target is $\H^2$.
Also for technical reasons, we work with maps $\phi$ that have finite mass, i.e.,
\[
\int_{\R^2} |\phi - Q|^2 dx < \infty
\]
where $Q \in M$ is a fixed point of $M \in \{\Sp^2, \H^2\}$. This quantity is not scale
invariant but is preserved by the flow.

\begin{thm}[Main result] \label{thm:maintheorem}
Let $I = [t_0, t_1] \subset \R$ with $t_0 < t_1$ and let $0 < \varepsilon \ll 1$.
If $\dot{H}^1 \cap \dot{H}^3 \ni \phi_0 = \phi(t_0) : \R^2 \to M$, $M \in \{\Sp^2, \H^2\}$, and
$\phi$ is a solution of \eqref{SM} on $I$ with finite mass and energy $E(\phi) < \Ec$ satisfying
$\lVert \nabla \phi \rVert_{\ell^2 L^4_{t,x}(I \times \R^2)} \leq \varepsilon \ll 1$,
then there exists a time $0 < T = T(\varepsilon)$ such that \eqref{SM}
has a unique solution in $L^\infty_t([t_0, t_1 + T] : \dot{H}^1 \cap \dot{H}^3)$.
\end{thm}
In fact, we obtain much more precise control on the solution $\phi$ than is indicated here;
see \S \ref{sec:mainresult} for our main technical result, which implies Theorem \ref{thm:maintheorem}.
Note that by time divisibility Theorem \ref{thm:maintheorem} also applies to solutions with large $\ell^2L^4$ norm.

\section{The caloric gauge}

Let $I \times \R^2 \ni (t, x) \mapsto \phi(t, x) \in M$ with $M \in \{\Sp^2, \H^2\}$ be a smooth
Schr\"odinger map.
A gauge choice may be represented by the map $e(t, x)$ in the diagram
\[
\begin{CD}
\R^2 \times \C @>e>> \phi^* TM @>>> TM \\
@AA\psi_\alpha A @AA\partial_\alpha \phi A @VV\pi V\\
I \times \R^2 @>id>> I \times \R^2 @>\phi >> M
\end{CD}
\]
Here $\psi_\alpha = e^* \partial_\alpha \phi$ denotes the vector $\partial_\alpha \phi$ written
with respect to the choice of orthonormal frame after canonically identifying $\R^2$ with $\C$.
The complex structure on $M$ pulls back to multiplication by $i$, and
the Levi-Civita connection pulls back to the covariant derivatives $D_\alpha := \partial_\alpha + i A_\alpha$, 
which generate curvatures $F_{\alpha \beta} := \partial_\alpha A_\beta - \partial_\beta A_\alpha$. 
Orthonormality of the frame ensures $A_\alpha \in \R$.
The zero-torsion property of the connection enforces the compatibility condition $D_\alpha \psi_\beta = D_\beta \psi_\alpha$. 
Using the fact that $M$ has constant curvature $\mu$, one may calculate directly that 
$F_{\alpha \beta} = \mu \Im(\bar{\psi}_\beta \psi_\alpha)$.
According to our conventions, $\mu = +1$ corresponds to the case of the sphere and $\mu = -1$
to that of the hyperbolic plane.
For any map $\phi$ and any choice of frame $e(t, x)$ it therefore holds that
\[
F_{\alpha \beta} = \mu \Im(\bar{\psi}_\beta \psi_\alpha) \quad \quad \text{and} \quad \quad
D_\alpha \psi_\beta = D_\beta \psi_\alpha
\]
These relations are preserved by the gauge transformations
\begin{equation}\label{gauge-freedom}
\phi \mapsto e^{-i \theta} \phi
\quad \quad
A \mapsto A + d \theta
\end{equation}
where $\theta(t, x)$ is a fast-decaying real-valued function.
This gauge invariance corresponds precisely to the freedom we that have in the choice of frame $e(t, x)$.

Here and throughout we use $\partial_0$ and $\partial_t$ interchangeably.
We also adopt the convention that Greek indices are allowed to assume values
from the set $\{-1, 0, 1, 2\}$, whereas Latin indices are restricted to $\{1, 2\}$, corresponding
only to spatial variables. Repeated Latin indices indicate an implicit sum over the spatial variables.
The case $\alpha = -1$ will be discussed below.

The lift of the Schr\"odinger map equation to the level of frames is
\begin{equation} \label{SM-freeze}
\psi_t = i D_j \psi_j
\end{equation}
To get an evolution equation for $\psi_j$, we covariantly differentiate with $D_k$
and use the compatibility condition $D_k \psi_t = D_t \psi_k$ to obtain
$D_t \psi_k = i D_k D_j \psi_j$. Next we commute $D_k$ and $D_j$ using the curvature
relation and we invoke the compatibility condition once more. This results in a covariant Schr\"odinger
evolution equation for $\psi_k$.
The whole gauge field system is
\begin{equation} \label{SM-gauge}
\begin{cases}
D_t \psi_k &= i D_j D_j \psi_k + F_{jk} \psi_j \\
F_{01} &= \mu \Re(\bar{\psi}_1 D_j \psi_j) \\
F_{02} &= \mu \Re(\bar{\psi}_2 D_j \psi_j) \\
F_{12} &= \mu \Im(\bar{\psi}_2 \psi_1) \\
D_1 \psi_2 &= D_2 \psi_1
\end{cases}
\end{equation}
Note that $\psi_t$ does not appear in this formulation, but can be recovered from the $\psi_j$
by using \eqref{SM-freeze}.
The system \eqref{SM-gauge} enjoys gauge freedom, which must be eliminated in order to
obtain a well-defined flow. The frame approach was first used to study Schr\"odinger map
wellposedness problems in \cite{ChShUh00}, though the system was formulated at the level
of frames at least as early as \cite{MaPaSo94}. In \cite{ChShUh00}, the gauge freedom
in \eqref{SM-gauge} is eliminated by choosing the Coulomb gauge condition.
We eliminate the gauge freedom by instead imposing the caloric gauge condition.
The caloric gauge was originally introduced in the setting of wave maps \cite{Tao04},
then subsequently applied to the Schr\"odinger map problem in \cite{BeIoKeTa11}.

The construction of the caloric gauge is most easily carried out at the level of maps, though in principle
one could carry it out entirely at the level of frames.
Let $\phi(t, x)$ be a map into $M$ defined on $I \times \R^2$
with energy $E_0 := E(\phi) = E(\phi(t))$.
We evolve, for each fixed $t_0 \in I$,
the map $\phi(t_0, x)$ under harmonic map heat flow, which is the gradient flow
associated to the energy \eqref{energy}:
\begin{equation} \label{HMHF}
\partial_s \phi = \Delta \phi + \mu |\partial_x \phi|_\mu^2 \phi
\end{equation}
We also use $\phi$ to denote the extension $\phi(s, t, x)$ along this flow.
Provided that the mass of $\phi(t_0)$ is finite and the energy of $\phi(t_0)$ is less than that of the ground state,
i.e., $E_0 < 4 \pi$ when the target is the sphere and
$E_0 < \infty$ when the target is the hyperboloid, the flow is well-defined and trivializes
as $s \to \infty$, sending all of $\R^2$ to a single point $Q \in M$.
To construct the caloric gauge, choose an orthonormal frame at $Q$,
pull it back at $s = \infty$, and finally pull it back along the heat flow using parallel transport.
This construction is unique modulo the one degree of freedom in the choice of frame at $Q$.
The validity of this construction up to the ground state and several related quantitative estimates
are established in \cite{Sm09}, extending the work initiated for $M = \H^2$ in \cite{Tao04, T4}.

At the level of gauges, \eqref{HMHF} assumes the form
\begin{equation} \label{HMHF-gauge}
\begin{cases}
D_s \psi_k &= D_j D_j \psi_k - i F_{jk} \psi_j \\
F_{-1,1} &= \mu \Im(\bar{\psi}_1 D_j \psi_j) \\
F_{-1,2} &= \mu \Im(\bar{\psi}_2 D_j \psi_j) \\
F_{12} &= \mu \Im(\bar{\psi}_2 \psi_1) \\
D_1 \psi_2 &= D_2 \psi_1
\end{cases}
\end{equation}
Here we use $-1$ to denote the $s$-time variable. We also may introduce
\begin{equation} \label{HF-freeze}
\psi_s = D_j \psi_j
\end{equation}
in analogy with $\psi_t$ given by \eqref{SM-freeze}.
Pulling back by parallel transport in the $s$ direction corresponds to taking
$A_s \equiv 0$. Then the $F_{-1, j}$ equations reduce to transport
equations for $A_\alpha$, and we may define $A_\alpha$ at finite times by integrating
back from infinity:
\begin{equation} \label{A}
A_\alpha(s) = - \int_s^\infty \Im(\bar{\psi}_\alpha D_j \psi_j)(s^\prime) ds^\prime
\end{equation}
Note that \eqref{A} is valid not only for $\alpha \in \{1, 2\}$, but also for $\alpha \in \{-1, 0\}$.

From \cite[Theorem 7.4]{Sm09}, we have several energy-type bounds for
the connection coefficients $A_x$:
\begin{equation} \label{A-energy}
\begin{split}
\sup_{s > 0} s^{\frac{k+1}{2}} \lVert \partial_x^k A_x(s) \rVert_{L^\infty_x} &\lesssim 1 
\\
\sup_{s > 0} s^{\frac{k}{2}} \lVert \partial_x^k A_x(s) \rVert_{L^2_x} &\lesssim 1
\\
\int_0^\infty s^{\frac{k-1}{2}} \lVert \partial_x^k A_x(s) \rVert_{L^\infty_x} ds &\lesssim 1
\\
\int_0^\infty s^{\frac{k-1}{2}} \lVert \partial_x^{k+1} A_x(s) \rVert_{L^2_x} ds &\lesssim 1
\end{split}
\end{equation}
Also, from \cite[Corollary 7.5]{Sm09}, we have energy-type bounds for
the gauge fields $\psi_x$:
\begin{equation} \label{psi-energy}
\begin{split}
\sup_{s > 0} s^{\frac{k}{2}} \lVert \partial_x^{k-1} \psi_x \rVert_{L^\infty_x} &\lesssim 1
\\
\sup_{s > 0} s^{\frac{k-1}{2}} \lVert \partial_x^{k-1} \psi_x \rVert_{L^2_x} &\lesssim 1
\\
\int_0^\infty s^{k-1} \lVert \partial_x^{k-1} \psi_x \rVert_{L^\infty_x}^2 ds &\lesssim 1
\\
\int_0^\infty s^{k-1} \lVert \partial_x^k \psi_x \rVert_{L^2_x}^2 ds &\lesssim 1
\end{split}
\end{equation}
In addition to \eqref{psi-energy}, we have analogous estimates
when one replaces $\partial_x \psi_x$ with $\psi_s$, $\partial_x^2$ with $\partial_s$,
and/or $\partial_x$ with $D_x$.
The constants in \eqref{A-energy} and \eqref{psi-energy} are allowed to depend upon
$k$ and upon the energy $E_0$.

We conclude by noting that the main evolution equations of 
\eqref{SM-gauge} and \eqref{HMHF-gauge} may respectively
be rewritten as
\begin{equation} \label{NLSH}
\begin{split}
(i \partial_t + \Delta) \psi_m &= \cN_m \\
(\partial_s - \Delta) \psi_\alpha &= U_\alpha
\end{split}
\end{equation}
where
\begin{equation} \label{NLSHN}
\begin{split}
\cN_m &:= - 2 i A_j \partial_j \psi_m - i (\partial_j A_j) \psi_m + (A_t + A_x^2) \psi_m - i \mu \psi_j \Im(\bar{\psi}_j \psi_m) \\
U_\alpha &:= 2i A_j \partial_j \psi_\alpha + i (\partial_j A_j) \psi_\alpha - A_x^2 \psi_\alpha + i \mu \psi_j \Im(\bar{\psi}_j \psi_\alpha)
\end{split}
\end{equation}
We assume that we are using the caloric gauge, which is why $A_{s} \equiv 0$ does not explicitly appear in the 
\eqref{NLSHN} expression for $U_\alpha$.

\section{Function spaces}

The main function spaces that we use were first introduced in their present form in \cite{BeIoKeTa11},
in which also is found a discussion of their development.
In particular, our $X_k$ and $Y_k$ spaces correspond to the $G_k$ and $N_k$ function spaces
of \cite{BeIoKeTa11}; we have no need for and therefore do not introduce the auxiliary function
space $F_k$ of \cite{BeIoKeTa11}.

\begin{lem}[Strichartz estimate]
Let $f \in L_x^2(\R^2)$ and $k \in \Z$.  Then the Strichartz estimate
\[
\lVert e^{i t \Delta} f \rVert_{L_{t,x}^4} \lesssim \lVert f \rVert_{L_x^2}
\]
holds, as does
the maximal function bound
\[
\lVert e^{i t \Delta} P_k f \rVert_{L_x^4 L_t^\infty} \lesssim
2^{\frac{k}{2}} \lVert f \rVert_{L_x^2}
\]
\end{lem}
The first bound is the original Strichartz estimate \cite{St77}
and the second follows from scaling.

For a unit length $\theta \in \Sp^1$, we denote by $H_\theta$ its orthogonal
complement in $\R^2$ with the induced Lebesgue measure.
Define the lateral spaces $L_\theta^{p,q}$ as those consisting of all 
measurable $f$ for which the norm
\[
\lVert h \rVert_{L_\theta^{p,q}} :=
\left[ \int_{\R} \left[
\int_{H_\theta \times \R} \lvert h(t, x_1 \theta + x_2)
\rvert^q dx_2 dt \right]^{\frac{p}{q}} dx_1 \right]^{\frac{1}{p}}
\]
is finite.
We make the usual modifications when $p = \infty$ or $q = \infty$.
For proofs of the following lateral Strichartz estimates, see \cite[\S3, \S7]{BeIoKeTa11}.
\begin{lem}[Lateral Strichartz estimates] \label{lem:Lateral}
Let $f \in L_x^2(\R^2)$, $k \in \Z$, and $\theta \in \mathbb{S}^1$.
Let $ 2 < p \leq \infty, 2 \leq q \leq \infty$ and $\frac1p + \frac1q = \frac12$. Then
\begin{align*}
\lVert e^{i t \Delta} P_{k, \theta} f \rVert_{L_\theta^{p, q}}
&\lesssim_{\phantom{p}}
2^{k(\frac2p - \frac12)} \lVert f \rVert_{L_x^2}, &p \geq q
\\
\lVert e ^{i t \Delta} P_k f \rVert_{L_\theta^{p,q}}
&\lesssim_p
2^{k(\frac2p - \frac12)} \lVert f \rVert_{L_x^2}, &p\leq q
\end{align*}
\end{lem}
In the Schr\"odinger map setting, local smoothing spaces were first used in \cite{IoKe06}
and subsequently in \cite{IoKe07b, BeIoKe07, Be08a, BeIoKeTa11}.
\begin{lem}[Local smoothing \cite{IoKe06, IoKe07b}] \label{lem:LS}
Let $f \in L_x^2(\R^2)$, $k \in \Z$, and $\theta \in \Sp^1$.  Then
\begin{equation*}
\lVert e^{i t \Delta} P_{k, \theta} f \rVert_{L_{\theta}^{\infty, 2}}
\lesssim
2^{-\frac{k}{2}} \lVert f \rVert_{L_x^2}
\end{equation*}
For $f \in L_x^2(\R^d)$, 
the maximal function space bound
\begin{equation*}
\lVert e^{i t \Delta} P_k f \rVert_{L_\theta^{2, \infty}} \lesssim
2^{\frac{k(d-1)}{2}} \lVert f \rVert_{L_x^2}
\end{equation*}
holds in dimension $d \geq 3$.
\end{lem}
In $d = 2$, the maximal function bound fails due to a logarithmic divergence.
This is overcome by exploiting Galilean invariance as in \cite{BeIoKeTa11}.
For $p, q \in [1, \infty]$, $\theta \in \Sp^1$, $\lambda \in \R$,
define $L_{\theta, \lambda}^{p,q}$ using the norm
\[
\lVert f \rVert_{L_{\theta, \lambda}^{p,q}} :=
\lVert T_{\lambda \theta}(f) \rVert_{L_\theta^{p,q}} =
\left[ \int_{\R} \left[
\int_{H_\theta \times \R}
\lvert f(t, (x_1 + t \lambda) \theta + x_2)\rvert^q dx_2 dt
\right]^{\frac{p}{q}} dx_1 \right]^{\frac{1}{p}}
\]
where $T_w$ denotes the Galilean transformation
\[
T_w(f)(t, x) := e^{-i x \frac w2} e^{-i t \frac{|w|^2}{4}} f(t, x + t w)
\]
With $W \subset \mathbb{R}$ finite we define the spaces $L_{\theta, W}^{p,q}$ by
\[
L_{\theta, W}^{p,q} := \sum_{\lambda \in W} L_{\theta, \lambda}^{p,q}, \quad \quad
\lVert f \rVert_{L_{\theta, W}^{p,q}} :=
\inf_{f = \sum_{\lambda \in W} f_\lambda} \sum_{\lambda \in W}
\lVert f_\lambda \rVert_{L_{\theta, \lambda}^{p,q}}
\]
For $k \in \Z$, $\cK \in \Z_{\geq 0}$, set
\[
W_k := \{ \lambda \in [-2^k, 2^k] : 2^{k + 2 \cK} \lambda \in \Z \}
\]
We work on a finite time interval
$[- 2^{2 \cK}, 2^{2 \cK}]$ in order to ensure that the $W_k$
are finite. This is still sufficient for global results provided all effective bounds
are uniform in $\cK$.
\begin{lem}[Local smoothing/maximal function estimates] \label{lem:LSMF}
Let $f \in L_x^2(\R^2)$, $k \in \Z$, and $\theta \in \Sp^1$.  Then
\[
\lVert e^{i t \Delta} P_{k, \theta} f \rVert_{L_{\theta, \lambda}^{\infty, 2}}
\lesssim 2^{-\frac{k}{2}} \lVert f \rVert_{L_{x}^2}, \quad \quad
\lvert \lambda \rvert \leq 2^{k - 40}
\]
and moreover, if $T \in (0, 2^{2 \cK}]$, then 
\[
\lVert 1_{[-T, T]}(t) e^{i t \Delta} P_k f
\rVert_{L_{\theta, W_{k + 40}}^{2, \infty}} \lesssim
2^{\frac{k}{2}} \lVert f \rVert_{L_{x}^2}
\]
\end{lem}
\begin{proof}
The first bound follows from Lemma \ref{lem:LS} via a Galilean boost.
The second is more involved and proven in \cite[\S 7]{BeIoKeTa11}.
\end{proof}

Let $I \subset \R$ be a time interval.  For $k \in \Z$, 
let $\Xi_k = \{ \xi \in \R^2 : \lvert \xi \rvert \in [2^{k-1}, 2^{k+1}] \}$.  
Let
\[
L_k^2(I) := \{ f \in L^2(I \times \R^2) : \supp\; \hat{f}(t, \xi) \subset I \times \Xi_k \}
\]
For $f \in L^2 (I \times \R^2)$, let
\[
\lVert f \rVert_{X_k^0(I)} :=
\lVert f \rVert_{L_t^\infty L_x^2} +
\lVert f \rVert_{L_{t,x}^4}
+ 2^{-\frac{k}{2}} \lVert f \rVert_{L_x^4 L_t^\infty}
+ 2^{-\frac{k}{2}} \sup_{\theta \in \Sp^1}
\lVert f \rVert_{L_{\theta, W_{k + 40}}^{2, \infty}}
\]
Define $X_k(I)$, $Y_k(I)$ as the normed spaces of functions in $L_k^2(I)$ for which the corresponding norms are finite:
\[
\begin{split}
\lVert f \rVert_{X_k(I)} :=& \;\lVert f \rVert_{X_k^0(I)}
+ 2 ^{-\frac{k}{6}} \sup_{\theta \in \Sp^1} \lVert f \rVert_{\Lts}
+ 2^{\frac{k}{6}} \sup_{\lvert j - k \rvert \leq 20} \sup_{\theta \in \mathbb{S}^1}
\lVert P_{j, \theta} f \rVert_{\Lst} \\
&+ 2^{\frac{k}{2}} \sup_{\lvert j - k \rvert \leq 20} \sup_{\theta \in \mathbb{S}^1}
\sup_{\lvert \lambda \rvert < 2^{k - 40}}
\lVert P_{j, \theta} f \rVert_{L_{\theta, \lambda}^{\infty, 2}} 
\\
\lVert f \rVert_{Y_k(I)} :=& \inf_{f = f_1 + f_2 + f_3 + f_4}
\lVert f_1 \rVert_{L_{t,x}^{\frac43}}
+ 2^{\frac{k}{6}} \lVert f_2 \rVert_{L_{\hat{\theta}_1}^{\frac32, \frac65}} \\
& + 2^{\frac{k}{6}} \lVert f_3 \rVert_{L_{\hat{\theta}_2}^{\frac32, \frac65}} 
+ 2^{-\frac{k}{2}} \sup_{\theta \in \mathbb{S}^1} 
\lVert f_4 \rVert_{L_{\theta, W_{k - 40}}^{1,2}}
\end{split}
\]
where $(\hat{\theta}_1, \hat{\theta}_2)$ denotes the canonical basis in $\mathbb{R}^2$.

These spaces are related via the following linear estimate, which is proved in \cite{BeIoKeTa11}.
\begin{prop}[Main linear estimate] \label{prop:MainLinearEstimate}
Assume $\cK \in \Z_{\geq 0}$, $I = [t_0, t_1] \subset (0, 2^{2 \cK}]$ and $k \in \Z$.  Then
for each $u_0 \in L^2$ that is frequency-localized to $\Xi_k$ and for any $h \in Y_k(I)$, the solution $u$ of
\begin{equation} \label{mainlinearestimate}
(i \partial_t + \Delta_x) u = h, \quad\quad u(t_0) = u_0
\end{equation}
satisfies
\begin{equation*}
\lVert u \rVert_{X_k(I)} \lesssim \lVert u(t_0) \rVert_{L_x^2} + \lVert h \rVert_{Y_k(I)}
\end{equation*}
\end{prop}
We conclude this section by recording some bilinear estimates.
\begin{lem}
For $k, j \in \Z$, $h \in L^2_{t,x}$, $f \in X_j(I)$, we have the following
inequalities under the given restrictions on $k, j$.
\begin{equation}\label{bilinear}
\lVert P_k(h f) \rVert_{Y_k(I)}
\lesssim
\begin{cases}
\phantom{2^{- \frac{|j-k|}{2}}} \lVert h \rVert_{L^2_{t,x}} \lVert f \rVert_{X_j(I)} & |j - k| \leq 80 \\
2^{- \frac{|j-k|}{2}} \lVert h \rVert_{L^2_{t,x}} \lVert f \rVert_{X_j(I)} & j \leq k - 80 \\
2^{- \frac{|j-k|}{6}} \lVert h \rVert_{L^2_{t,x}} \lVert f \rVert_{X_j(I)} & k \leq j - 80
\end{cases}
\end{equation}
\end{lem}
\begin{proof}
See \cite[Lemma 6.3]{BeIoKeTa11}.
\end{proof}

\section{The main result} \label{sec:mainresult}

In this section we state and outline the proof of our main technical result.

It is shown in \cite[Lemma 4.3]{Sm10} that
\[
\lVert \nabla \phi \rVert_{\ell^2 L^4_{t,x}(I \times \R^2)}
\sim
\lVert \psi_x \rVert_{\ell^2 L^4_{t,x}(I \times \R^2)}
\]
and so the small $\ell^2 L^4$ assumption of Theorem \ref{thm:maintheorem}
directly lifts to the gauge formulation: we take $0 < \varepsilon \ll 1$ such that
\begin{equation} \label{ScatteringSize}
\lVert \psi_x \rVert_{\ell^2 L^4_{t,x}(I \times \R^2)} \leq \varepsilon
\end{equation}

Pick $0 < \delta \ll 1$. A positive sequence $\{a_k\}_{k \in \Z}$ is said to be a 
frequency envelope provided that it belongs to $\ell^2$ and is slowly varying
in the sense that $a_k \leq a_j 2^{\delta |k - j|}$ for all $j, k \in \Z$.
Frequency envelopes satisfy the summation rules
\begin{align*}
\sum_{k^\prime \leq k} 2^{p k^\prime} a_{k^\prime}
&\lesssim (p - \delta)^{-1} 2^{p k} a_k
&p > \delta \\
\sum_{k^\prime \geq k} 2^{-p k^\prime} a_{k^\prime}
&\lesssim (p - \delta)^{-1} 2^{-p k} a_k
&p > \delta
\end{align*}
We absorb the $(p - \delta)^{-1}$ factor into the constant in applications since it only ever appears
$O(1)$ many times.

For $\sigma \in \Z_{\geq 0}$, and $I = [t_0, t_1]$, define the frequency envelopes
$b_k(\sigma), \alpha_k(\sigma)$, and $\beta_k(\sigma)$
via
\[
\begin{split}
b_k(\sigma) &= \sup_{j \in \Z} 2^{\sigma j} 2^{- \delta |k - j|}
\lVert P_{j} \psi_x \rVert_{X_j(I)} \\
\alpha_k(\sigma) &= \sup_{j \in \Z} 2^{\sigma j} 2^{- \delta |k - j|} \lVert P_j \psi_x \rVert_{L^4_{t,x}} \\
\beta_k(\sigma) &= \sup_{j \in \Z} 2^{\sigma j} 2^{- \delta |k - j|} \lVert P_j \psi_x(t_0) \rVert_{L^2_x}
\end{split}
\]
Set $b_k = b_k(0)$, $\alpha_k = \alpha_k(0)$, and $\beta_k = \beta_k(0)$ for short.
These envelopes satisfy
\[
\sum_k b_k^2 \sim \sum_k \lVert P_k \psi_x \rVert_{X_k(I)}^2
\]
and
\[
\sum_{k \in \Z} \alpha_k^2 \sim \lVert \psi_x \rVert_{\ell^2 L^4}^2,
\quad \quad
\sum_{k \in \Z} \beta_k^2 \sim \lVert P_j \psi_x(t_0) \rVert_{\ell^2 L^2_x}^2 \sim E_0^2
\]
For convenience, set 
\[
\upsilon_k(\sigma) := \alpha_k(\sigma) + \beta_k(\sigma)
\]

\begin{thm}[Main technical result] \label{thm:maintheorem-tech}
Let $I = [t_0, t_1] \subset \R$ with $t_0 < t_1$.
Let $\dot{H}^1 \cap \dot{H}^3 \ni \phi_0 = \phi(t_0) : \R^2 \to M$, $M \in \{\Sp^2, \H^2\}$, and
let $\phi$ be a solution of \eqref{SM} on $I$ with finite mass, with energy $E(\phi) < \Ec$, and
with caloric gauge representation $(\psi_\alpha, A_\alpha)$.
Let $0 < \delta, \varepsilon \ll 1$, $\sigma_1 \in \Z_{> 0}$, 
and let frequency envelopes $b_k(\sigma), \upsilon_k(\sigma)$ be defined as above. 
If \eqref{ScatteringSize} holds, then
\[
b_k(\sigma) \lesssim_{\varepsilon} \upsilon_k(\sigma)
\]
for $\sigma \in \{0, 1, \ldots, \sigma_1\}$.
\end{thm}
\begin{proof}
The proof is by a standard continuity argument where we make the bootstrap hypothesis
$b_k \leq \varepsilon^{-\frac12} \upsilon_k$. Next we apply $P_k$
to the covariant Schr\"odinger equation in \eqref{NLSH} and
apply the main linear estimate \eqref{mainlinearestimate}. 
This reduces the problem to controlling $P_k \cN_m$ in the $Y_k(I)$ spaces.
Part of $P_k \cN_m$ is perturbative, in that it can be bounded in $Y_k(I)$
by $\varepsilon^2 b_k$. This is proved in Lemma \ref{lem:Npert} below.
Remaining is the non-perturbative part of $P_k \cN_m$. In \S \ref{sec:closer}, we provide
two separate arguments that address how to deal with the non-perturbative part and close
the bootstrap.
In both arguments we need to control $X_k(I)$ bounds
of $P_k \psi_x$ along the heat flow; 
such bounds are established in \S \ref{sec:heatflowbounds}.
\end{proof}
\begin{rem}
In the proofs we work with the $\sigma = 0$ case, which is the critical case to establish.
The same proofs are valid for $\sigma = \sigma_1 > 0$ provided that in controlling the 
Littlewood-Paley decompositions we use the
$\sigma = \sigma_1$ frequency envelope for the highest frequency term and the $\sigma = 0$
frequency envelopes for the remaining terms.
See \cite[\S 7]{Sm11} for additional related remarks.
\end{rem}
Our main technical tool is the following bilinear estimate, established by the first author in \cite{Do12}.
Various precursors to this estimate appear in
\cite{BeIoKeTa11, Sm10, Sm11}. In \cite{Do12}, the estimate is established for the target $\Sp^2$,
but the proof also applies to the case where the target is $\H^2$.
\begin{thm}
If a solution $\psi_x$ of \eqref{SM-gauge} in the caloric gauge satisfies \eqref{ScatteringSize}
and has $L^2_x$ norm less than $E_0^{\frac12}$, then
\begin{equation} \label{bls}
\lVert (P_j \bar{\psi}_x(s)) (P_k \psi_x(\tilde{s})) \rVert_{L^2_{t,x}(I \times \R^2)}
\lesssim
2^{-\frac{|j - k|}{2}} \upsilon_j \upsilon_k (1 + s 2^{2j})^{-4} (1 + \tilde{s}2^{2k})^{-4}
\end{equation}
\end{thm}
Related to this are the following $L^2$-based estimates, 
established in the proof of \cite[Theorem 6.3]{Do12} (see equation (6.107) in that work).
\begin{lem} \label{lem:L2bounds}
If a solution $(\psi_x, A_\alpha)$ of \eqref{SM-gauge} in the caloric gauge satisfies \eqref{ScatteringSize}
and $\psi_x$ has $L^2_x$ norm less than $E_0^{\frac12}$, then
\[
\lVert A_x \rVert_{L^4_{t,x}}^2 + \lVert \psi_x \rVert_{L^4_{t,x}}^2 +
\lVert \partial_x A_x \rVert_{L^2_{t,x}}+ \lVert A_t \rVert_{L^2_{t,x}}
\lesssim \varepsilon^2
\]
where all norms are taken over the space-time slab $I \times \R^2$.
\end{lem}
We have the following technical lemma from \cite[Lemma 5.2]{Sm11}:
\begin{lem} \label{lem:YkL2}
Let $f \in L^2_{t,x}$. Then
\[
\lVert P_k(f \psi_m)\rVert_{Y_k(I)} \lesssim \lVert f \rVert_{L^2_{t,x}(I \times \R^2)} b_k
\]
\end{lem}
Combining Lemmas \ref{lem:YkL2} and \ref{lem:L2bounds} yields
\[
\lVert P_k \left[ - i (\partial_j A_j) \psi_m + (A_t + A_x^2) \psi_m - i \mu \psi_j \Im(\bar{\psi}_j \psi_m) \right] \rVert_{Y_k(I)}
\lesssim
\varepsilon^2 b_k
\]
which proves the
\begin{lem} \label{lem:Npert}
The term $\cN_m + 2i A_\ell \partial_\ell \psi_m$ is perturbative.
\end{lem}
We address the remaining non-perturbative term $2i A_\ell \partial_\ell \psi_m$ in \S \ref{sec:closer}.
Here we show that we can return from the gauge formulation to the map formulation
with the following
\begin{lem}
It holds that
\[
\lVert \nabla \phi \rVert_{\dot{H}^\sigma}^2 
\lesssim 
\sum_{k \in \Z} \sum_{\sigma^\prime = 0}^{2 \sigma - 1} b_k^2(\sigma^\prime)
\]
\end{lem}
\begin{proof}
The statement is far from optimal.
A stronger estimate is established in \cite[\S 4.6]{Sm10}, but under a certain smallness assumption.
The smallness assumption is not need, however, for the $\sigma = 0$ case, and this carries over
without modification. 
The same argument works for $\sigma > 0$
except for certain high-low frequency interactions for $\sigma > 0$.
To derive the claimed expression for $\sigma > 0$,
consider $\lVert \nabla \phi \rVert_{\dot{H}^\sigma}^2$ as an expression bilinear in $\nabla \phi$
and project the product to frequencies $\sim 2^k$. Then use a standard Littlewood-Paley decomposition,
enough integrations by parts, and Cauchy-Schwarz.
\end{proof}
Theorem \ref{thm:maintheorem-tech} combined with the preceding lemma
and the local result stated in Theorem \ref{thm:local} establish Theorem \ref{thm:maintheorem}.

\section{Bounds along the heat flow} \label{sec:heatflowbounds}

In this section we prove that $X_k$ bounds of $P_k \psi_x$ propagate along the heat flow and exhibit decay.
We make frequent use of the Duhamel representation
\begin{equation} \label{Duhamel}
\psi_m(s) = e^{s \Delta} \psi_m(0) + \int_0^s e^{(s - s^\prime)\Delta} U_m(s^\prime) ds^\prime
\end{equation}
with $U_m$ as in \eqref{NLSHN}.
For frequency envelope definitions, see \S \ref{sec:mainresult}.
\begin{thm} \label{thm:HF0}
If a solution $\psi_x$ of \eqref{SM-gauge} in the caloric gauge 
has $L^2_x$ norm less than $E_0^{\frac12}$, then
\begin{equation} \label{HF0}
\lVert P_k \psi_x(s) \rVert_{X_k(I)} \lesssim \upsilon_k (1 + s 2^{2k})^{-4}
\end{equation}
\end{thm}
\begin{proof}
The following three inequalities play a key role in the proof and will be established
in subsequent lemmas:
\begin{align}
\label{HF1}
\lVert P_k \psi_x(s) \rVert_{L^\infty_t L^2_x}
&\lesssim
\upsilon_k (1 + s 2^{2k})^{-4}
\\
\label{HF2}
\lVert P_k \left[ \psi_x(s) - e^{s \Delta} \psi_x(0) \right] \rVert_{L^2_{t,x}} 
&\lesssim
\upsilon_k (s^{-\frac12} + 2^k)^{-1} (1 + s 2^{2k})^{-4}
\\
\label{HF3}
\lVert (\partial_t - i \Delta) P_k \psi_x(s) \rVert_{L^2_{t,x}}
&\lesssim
\upsilon_k (s^{-\frac12} + 2^k) (1 + s 2^{2k})^{-4}
\end{align}

First we prove \eqref{HF0} for high modulations. Using \eqref{HF3}, we have
\[
\lVert P_{|\tau| \sim 2^{2j}, |\xi| \sim 2^k} \psi_x(s) \rVert_{L^2_{t,x}}
\lesssim
2^{-2j} (s^{-\frac12} + 2^k) \upsilon_k (1 + s 2^{2k})^{-4}
\]
Then \eqref{HF0} is established for $P_{|\tau| > 2^{2k + 20}, |\xi| \sim 2^k}$ when $s > 2^{-2k}$, $j > k + 10$
and for $P_{|\tau| > 2^{20} s^{-1}, |\xi| \sim 2^k}$, $s < 2^{-2k}$, $j > k + 10$ by appropriate Sobolev embeddings.
For the low modulation case, we expand $\psi_x$ using \eqref{Duhamel}. 
The bound for the linear flow follows from the translation invariance of the $X_k$ norms:
\[
\lVert e^{s \Delta} P_k \psi_x(0) \rVert_{X_k} \lesssim \upsilon_k (1 + s 2^{2k})^{-4}
\]
To obtain the bound for $P_k \left[ \psi_x(s) - e^{s \Delta} \psi_x(0) \right]$, we combine \eqref{HF2}
with appropriate Sobolev embeddings, separately considering $s > 2^{-2k}$ and $s < 2^{-2k}$.

This completes the proof for the case $I = \R$. We return to the general case after
establishing several lemmas.
\end{proof}

In \cite{Do12} it is established that $L^4$ control propagates along the heat flow:
\begin{equation} \label{DoL4}
 \| P_{k} \psi_{x}(s) \|_{L_{t,x}^{4}} \lesssim \upsilon_k(1 + s 2^{2k})^{-4}
\end{equation}
This is complemented by the following result.
\begin{lem}
It holds that
\begin{equation} \label{PkAL4}
\lVert P_k A_x(s) \rVert_{L^4_{t,x}} \lesssim \upsilon_k (1 + s 2^{2k})^{-4}
\end{equation}
\end{lem}
\begin{proof}
If $s < 2^{-2k}$, then
\[
\begin{split}
\lVert P_k A_x(s) \rVert_{L^4_{t,x}}
\lesssim &\;
\sum_{j \geq k + 5} \int_s^\infty 2^k \lVert P_j \psi_x \rVert_{L^4_{t,x}} 
\left( \lVert \nabla \psi_x \rVert_{L^2_x} + \lVert \psi_x \rVert_{L^\infty_x} \lVert A_x \rVert_{L^2_x} \right) ds^\prime
\\
& 
+\int_s^\infty \lVert P_{k - 5 \leq \cdot \leq k + 5} \psi_x \rVert_{L^4_{t,x}} \lVert \psi_x \rVert_{L^\infty_x}
\left( 2^k + \lVert A_x \rVert_{L^\infty_x} \right) ds^\prime
\\
& 
+\sum_{j \leq k } \int_s^\infty 2^{\frac{j}{2}} \lVert P_j \psi_x \rVert_{L^4_{t,x}} \lVert \psi_x \rVert_{L^\infty_x}
\lVert P_{> k-5} A_x \rVert_{L^4_x} ds^\prime
\end{split}
\]
and the right hand side is bounded by $\upsilon_k$.

On the other hand, if $s \geq 2^{-2k}$, then
\[
\begin{split}
\lVert P_k A_x(s) \rVert_{L^4_{t,x}}
\lesssim &\;
\sum_{j \geq k - 5} \int_s^\infty \lVert P_j \psi_x \rVert_{L^4_{t,x}} 
\left( \lVert \nabla \psi_x \rVert_{L^\infty_x} + \lVert \psi_x \rVert_{L^\infty_x} \lVert A_x \rVert_{L^\infty_x} \right) ds^\prime
\\
& 
\sum_{j \leq k - 5} \int_s^\infty 2^{\frac{j}{2}} \lVert P_j \psi_x \rVert_{L^4_{t,x}}
\left( \lVert \nabla P_{> k - 5} \psi_x \rVert_{L^4_x} + \lVert \psi_x \rVert_{L^\infty_x} \lVert P_{>k-5} A_x \rVert_{L^4_x} \right) ds^\prime
\end{split}
\]
and the right hand side is bounded by $\upsilon_k(1 + s 2^{2k})^{-4}$.
\end{proof}

\begin{lem}
It holds that
\[
\lVert P_k \psi_x(s) \rVert_{L^\infty_t L^2_x}
\lesssim
\upsilon_k (1 + s 2^{2k})^{-4}
\]
\end{lem}
\begin{proof}
Using representation \eqref{Duhamel}, we have
by translation invariance that
\[
\lVert e^{s \Delta} P_k \psi_x(0) \rVert_{L^\infty_t L^2_x} 
\lesssim 
\upsilon_k (1 + s 2^{2k})^{-4}
\]
for the linear term.
Next we bound the nonlinear Duhamel term in $L^2_x$.
Using the Littlewood-Paley trichotomy we obtain
\[
\begin{split}
&\lVert P_k \int_0^s e^{(s - s^\prime) \Delta} U_x(s^\prime) ds^\prime \rVert_{L^2_x}
\lesssim
\\
&\quad
K_1 \| P_{k - 5 \leq \cdot \leq k + 5} \psi_{x} \|_{L_{s}^{\infty} L_{x}^{2}}
+ K_2 
\left[
\sum_{j \leq k} 2^{-|j - k|} \| P_{j} \psi_{x} \|_{L_{s}^{\infty} L_{x}^{2}}  +
\sum_{j \geq k + 5} 2^{- |k - j| } \| P_{j} \psi_{x} \|_{L_{s}^{\infty} L_{x}^{2}} 
\right]
\end{split}
\]
where here
\[
K_1 :=
\| \nabla \cdot A \|_{L_{s}^{1} L_{x}^{\infty}} 
+ 2^{k} \left(\int_{0}^{\infty} e^{-s 2^{2k}} ds\right)^{\frac12}
\| A \|_{L_{s}^{2} L_{x}^{\infty}} + \| A_{x} \|_{L_{s}^{2} L_{x}^{\infty}}^{2} 
+ \| \psi_{x} \|_{L_{s}^{2} L_{x}^{\infty}}^{2}
\]
and
\[
K_2 :=
\| \Delta A \|_{L_{s}^{1} L_{x}^{2}} +
2^k \left(\int_{0}^{\infty} e^{-s 2^{2k}} ds\right)^{\frac12}  \| \nabla A \|_{L_{s,x}^{2}} +
\| \nabla (A_{x} + \psi_{x}) \|_{L_{s,x}^{2}} \| A_{x} + \psi_{x} \|_{L_{s}^{2} L_{x}^{\infty}}
\]
In view of the estimates \eqref{A-energy} and \eqref{psi-energy}, we have $K_1 + K_2 \lesssim_{E_0} 1$, 
and therefore
by partitioning $[0, \infty)$ into finitely many pieces we may arrange 
$K_1 + K_2 \leq \epsilon \ll 1$ on each piece. Iterating then yields
$\lVert P_k \psi_x(s) \rVert_{L^2_x} \lesssim \upsilon_k$.

To obtain decay in $s$, we make the bootstrap assumption
$\lVert P_k \psi_x(s) \rVert_{L^2_x} \leq C \upsilon_k (1 + s 2^{2k})^{-4}$.
For $\delta > 0$, $N \in \Z_{\geq 0}$, it holds that
\[
e^{-(s - s^\prime)2^{2k}} \lesssim_{\delta, N} (1 + s 2^{2k})^{-N},
\quad \quad
s^\prime < (1 - \delta) s
\]
and so the integral over $[0, (1 - \delta)s]$ is controlled as follows:
\[
\lVert P_k \int_0^{(1 - \delta) s} e^{(s - s^\prime)\Delta} U_x(s^\prime) ds^\prime \rVert_{L^2_x}
\lesssim
\upsilon_k (1 + s 2^{2k})^{-4}
\]
Over $[(1 - \delta)s, s]$, we have from the weighted estimates \eqref{A-energy}, \eqref{psi-energy} that
\[
\| s \nabla \cdot A \|_{L^{\infty}_{s,x}} 
+ 2^{k} \left(\int_{0}^{\infty} e^{-s 2^{2k}} ds\right)^{\frac12}
\| s^{\frac12} A \|_{L^{\infty}_{s,x}} + \| s^\frac12 A_{x} \|_{L^\infty_{s,x}}^2
+ \| s^{\frac12} \psi_{x} \|_{L^\infty_{s,x}}^{2}
\lesssim_{E_0} 1
\]
and
\[
\begin{split}
\| s \Delta A \|_{L^\infty_s L^2_x} &+
2^k \left(\int_{0}^{\infty} e^{-s 2^{2k}} ds\right)^{\frac12}  \| s^\frac12 \nabla A \|_{L^\infty_s L^2_x} 
\\
&+
\| s^\frac12 \nabla (A_{x} + \psi_{x}) \|_{L^\infty_s L^2_x} + \| s^\frac12 (A_{x} + \psi_{x}) \|_{L^\infty_{s,x}}
\lesssim_{E_0} 1
\end{split}
\]
Choosing $\delta = \delta(E_0) > 0$ sufficiently small closes the argument.
\end{proof}

\begin{lem}
It holds that
\[
\lVert P_k \left[ \psi_x(s) - e^{s \Delta} \psi_x(0) \right] \rVert_{L^2_{t,x}} 
\lesssim
\upsilon_k (s^{-\frac12} + 2^k)^{-1} (1 + s 2^{2k})^{-4}
\]
\end{lem}
\begin{proof}
To prove the estimate, we need to bound the nonlinear Duhamel term of $\psi_x$ in $L^2_{t,x}$.
We proceed term by term.

First, we have
\[
\begin{split}
\lVert 2^k \int_0^s e^{(s - s^\prime) \Delta} P_k (A \psi_x) ds^\prime \rVert_{L^2_{t,x}}
\lesssim &\;
2^k \min(s, 2^{-2k}) \| P_{k - 5 \leq \cdot \leq k + 5} \psi_{x} \|_{L_{t,x}^{4}} \| A_{x} \|_{L_{t,x}^{4}} 
\\ &
+ 2^{k} \min(s, 2^{-2k}) \sum_{j \leq k} 2^{\frac{j}{2}} \| P_{k} A_{x} \|_{L_{t}^{4} L_{x}^{2}} \| P_{j} \psi_{x} \|_{L_{t,x}^{4}} 
\\ &
+ 2^{2k} \min(s, 2^{-2k}) \sum_{j > k + 5} 2^{-j} \| P_{j} A_{x} \|_{L_{t}^{4} L_{x}^{\frac43}} \| P_{j} \psi_{x} \|_{L_{t,x}^{4}} 
\end{split}
\]
The right hand side is controlled by $\min(s^{\frac12}, 2^{-k}) \upsilon_k (1 + s 2^{2k})^{-4}$.

\begin{rem}
Once again the $(1 + s 2^{2k})^{-4}$ gain comes from the decay of $e^{(s - s') \Delta}$ when $s' < \frac{s}{2}$ 
and from
$\| P_{k} \psi_{x} \|_{L_{t,x}^{4}} \lesssim (1 + s 2^{2k})^{-4}$ and the decay of $P_{k} A_{x}$
for larger $s^\prime$.
\end{rem}

Next, we have
\[
\begin{split}
\lVert \int_0^s e^{(s - s^\prime)\Delta} P_k( (\nabla \cdot A) \psi_x) ds^\prime \rVert_{L^2_{t,x}}
\lesssim &\;
\min(2^{-k}, s^{\frac12}) \| P_{k} \psi_{x} \|_{L_{t,x}^{4}} \| \nabla \cdot A_{x} \|_{L_{s}^{2} L_{t,x}^{4}}
\\ &
+ \min(s^{\frac12}, 2^{-k}) \| P_{k}(\nabla \cdot A) \|_{L_{s}^{2} L_{t}^{4} L_{x}^{2}} \sum_{j \leq k} 2^{\frac{j}{2}} \| P_{j} \psi_{x} \|_{L_{t,x}^{4}}
\\ &
+ \sum_{j > k + 5} 2^{\frac{k}{2}} \| P_{j} (\nabla \cdot A) \|_{L_{s}^{2} L_{t}^{4} L_{x}^{2}} \| P_{j} \psi_{x} \|_{L^2_s L_{t,x}^{4}}
\end{split}
\]
and again we obtain the desired control on the right hand side.

Finally, we have
\[
\begin{split}
&\| \int_0^s e^{(s - s') \Delta} P_{k} ((A_{x}^{2} + \psi_{x}^{2}) \psi_{x}) ds' \|_{L_{t,x}^{2}}
\\
&\quad \quad \quad \quad \lesssim
\| P_{k} \psi_{x} \|_{L_{t,x}^{4}} \| (A_{x} + \psi_{x}) \|_{L_{t,x}^{4}} \| (A_{x} + \psi_{x}) \|_{L_{s}^{2} L_{x}^{\infty}}
\\ & \quad \quad \quad \quad \quad
+ \sum_{j > k + 5} 2^{k} \| P_{j} \psi_{x} \|_{L_{t,x}^{4}} \| P_{j} (A_{x} + \psi_{x}) \|_{L_{s,x}^{2}} \| A_{x} + \psi_{x} \|_{L_{t,x}^{4}}
\\ & \quad \quad \quad \quad \quad
+ \sum_{j \leq k} 2^{\frac{j}{2}} \| P_{j} \psi_{x} \|_{L_{t,x}^{4}} \| P_{> k} (A_{x} + \psi_{x}) \|_{L_{s}^{2} L_{x}^{4}} \| A_{x} + \psi_{x} \|_{L_{t,x}^{4}}
\end{split}
\]
and the desired bound follows.
\end{proof}

\begin{lem}
It holds that
\[
\lVert (\partial_t - i \Delta) P_k \psi_x(s) \rVert_{L^2_{t,x}}
\lesssim
\upsilon_k (s^{-\frac12} + 2^k) (1 + s 2^{2k})^{-4}
\]
\end{lem}
\begin{proof}

We take advantage of
the compatibility condition $D_t \psi_x = D_x \psi_t$, which upon expansion reads
\begin{equation} \label{txcompexp}
\partial_t \psi_x(s) = \partial_x \psi_t(s) + i A_x \psi_t(s) - i A_t \psi_x(s)
\end{equation}
The terms on the right hand side we then expand using the Duhamel representation \eqref{Duhamel}.
Our strategy will be to bound each of $A_t \psi_x$ and $A_x \psi_t$ directly,
followed by $\partial_x \psi_t - i \Delta \psi_x$.

\textbf{The term $A_t \psi_x$}

For $A_t \psi_x$, we have
\[
\begin{split}
\lVert P_k (A_t \psi_x) \rVert_{L^2_{t,x}}
\lesssim &
\sum_{j \geq k + 5} 2^{k} \| P_{j} A_{t} \|_{L_{t}^{4} L_{x}^{\frac43}} \| P_{j} \psi_{x} \|_{L_{t,x}^{4}}
+ \| P_{k - 5 \leq \cdot \leq k + 5} \psi_{x} \|_{L_{t,x}^{4}} \| A_{t} \|_{L_{t,x}^{4}}
\\ &\;
+ \sum_{j \leq k - 5} 2^{\frac{j}{2}} 2^{k/2} \| P_{j} \psi_{x} \|_{L_{t,x}^{4}} \| A_{t} \|_{L_{t}^{4} L_{x}^{\frac43}}
\end{split}
\]
As
\[
\| A_{t} \|_{L_{s}^{2} L_{t,x}^{4}} + \| s^{\frac12} A_{t} \|_{L_{t,x}^{4}} \lesssim \| \psi_{t} \|_{L_{s}^{2} L_{t,x}^{4}} (\| s \partial_{x} \psi_{x} \|_{L_{s,x}^{\infty}} + \| s^{\frac12} A_{x} \|_{L_{s,x}^{\infty}} \| s^{\frac12} \psi_{x} \|_{L_{s,x}^{\infty}}) \lesssim \varepsilon
\]
and
\[
\| A_{t} \|_{L_{t}^{4} L_{x}^{\frac43}} \lesssim \| \psi_{t} \|_{L_{s}^{2} L_{t,x}^{4}} (\| \partial_{x} \psi_{x} \|_{L_{s,x}^{2}} + \| A_{x} \|_{L_{s}^{2} L_{x}^{\infty}} \| \psi_{x} \|_{L_{s}^{\infty} L_{x}^{2}}) \lesssim \varepsilon
\]
thanks to \cite[Lem.~6.7, Cor.~6.8, Thm.~6.9]{Do12},
we conclude $\lVert P_k (A_t \psi_x) \rVert_{L^2_{t,x}} \lesssim (s^{-\frac12} + 2^k) \upsilon_k (1 + s 2^{2k})^{-4}$.

\textbf{The term $A_x \psi_t$}

Invoking \eqref{PkAL4}, we have
\[
\lVert P_{k - 5 \leq \cdot \leq k + 5} A_x \rVert_{L^4_{t,x}}
\lVert \psi_t \rVert_{L^4_{t,x}}
\lesssim
s^{-\frac12} \upsilon_k \varepsilon (1 + s 2^{2k})^{-4}
\]

For the remaining terms, we use the following Duhamel expansion of $\psi_t$:
\[
\psi_t(s) = i e^{s \Delta} \partial_\ell \psi_\ell(0) - e^{s \Delta}(A_\ell \psi_\ell)(0)
+ \int_0^s e^{(s - s^\prime) \Delta} U_t(s^\prime) ds^\prime
\]
For the first term from the expansion, we have
\[
\begin{split}
\| P_{k - 5 \leq \cdot \leq k + 5} i e^{s \Delta} \partial_{l} \psi_{l}(0) \|_{L_{t,x}^{4}} \| A_{x} \|_{L_{t,x}^{4}} 
&+ \sum_{j \geq k + 5} 2^{k} \| P_{j} i e^{s \Delta} \partial_{l} \psi_{l}(0) \|_{L_{t,x}^{4}} \| P_{j} A_{x} \|_{L_{t}^{4} L_{x}^{\frac43}} 
\\
&\lesssim s^{-\frac12} \upsilon_k (1 + s 2^{2k})^{-4}
\end{split}
\]
For the second, we use
\[
 \| e^{s \Delta} P_{j}(A_{l} \psi_{l})(0) \|_{L_{t}^{4} L_{x}^{2}} 
 \lesssim 2^{-\frac{j}{2}} s^{-\frac12} \| \psi_{x} \|_{L_{t,x}^{4}} \| A_{x} \|_{L_{x}^{2}} 
 \lesssim 2^{-\frac{j}{2}} s^{-\frac12} \varepsilon
\]
to conclude
\[
\begin{split}
\sum_{j \leq k} 2^{\frac{j}{2}} \| P_{j} A_{x} \|_{L_{t,x}^{4}} \| P_{> k - 5} e^{s \Delta} (A_{l} \psi_{l})(0) \|_{L_{t}^{4} L_{x}^{2}} 
&+ \sum_{j > k + 5} 2^{k/2} \| P_{j} A_{x} \|_{L_{t,x}^{4}} \| P_{j} e^{s \Delta} (A_{l} \psi_{l})(0) \|_{L_{t}^{4} L_{x}^{2}} 
\\
&\lesssim s^{-\frac12} \upsilon_k (1 + s 2^{2k})^{-4}
\end{split}
\]
This leaves the nonlinear Duhamel term.
We can control it in $L^4_t L^{\frac43}_x$ via
\[
 \| \int_{0}^{s} e^{(s - s') \Delta} U_{t}(s') ds' \|_{L_{t}^{4} L_{x}^{\frac43}} 
 \lesssim \| \psi_{t} \|_{L_{s}^{2} L_{t,x}^{4}} (\| A \|_{L_{s}^{\infty} L_{x}^{2}} + \| \nabla \cdot A \|_{L_{s,x}^{2}} + \| A_{x} \|_{L_{s,x}^{4}}^{2} + \| \psi_{x} \|_{L_{s,x}^{4}}^{2}) 
 \lesssim \varepsilon
\]
Then
\[
 \sum_{j \leq k - 5} \| P_{k - 5 \leq \cdot \leq k + 5} 
 (\int_{0}^{s} e^{(s - s') \Delta} U_{t}(s') ds') 
 \|_{L_{t}^{4} L_{x}^{2}} 2^{\frac{j}{2}} \| P_{j} A_{x} \|_{L_{t,x}^{4}} 
 \lesssim 
 2^{k} \upsilon_k (1 + s2^{2k})^{-4}
\]
and
\[
\sum_{j \geq k + 5} 2^{k} \| P_{j}(\int_{0}^{s} e^{(s - s') \Delta} U_{t}(s') ds') \|_{L_{t}^{4} L_{x}^{\frac43}} 
\| P_{j} A_{x} \|_{L_{t,x}^{4}} 
\lesssim 
s^{-\frac12} \upsilon_k (1 + s 2^{2k})^{-4}
\]
We conclude $\|  P_{k} (A_{x} \psi_{t}) \|_{L_{t,x}^{2}} \lesssim (s^{-\frac12} + 2^{k}) \upsilon_k (1 + s 2^{2k})^{-4}$.
\begin{rem}
It is possible to prove $2^{-8k} s^{-4}$ decay of $\| P_{> k - 5} \psi_{t} \|_{L_{s}^{2} L_{t,x}^{4}}$ by using the usual bootstrap argument.
\end{rem}

\textbf{The term $\partial_x \psi_t - i \Delta \psi_x$}

By \eqref{HF2}, we have
\[
 \| -i \Delta \int_{0}^{s} e^{(s - s') \Delta} P_{k}(U_{x}(s')) ds' \|_{L_{t,x}^{2}} \lesssim 2^{k} \upsilon_k (1 + s 2^{2k})^{-4}
\]
Next, taking advantage of $\int_0^s e^{-(s - s^\prime)2^{2k}} ds^\prime \lesssim 2^{-2k}$, we write
\[
 \| \partial_{x} \int_{0}^{s} e^{(s - s') \Delta} P_{k}(U_{t}(s')) ds' \|_{L_{t,x}^{2}} 
 \lesssim
 K_1 + K_2 + K_3 + K_4 + K_5
\]
where here
\[
\begin{split}
K_1 &:= 
2^{k} \| s^{\frac12} P_{k}(A_{x} \psi_{t}) \|_{L_{s}^{\infty} L_{t,x}^{2}}
\\
K_2 &:=
2^{\frac32 k} \sum_{j \geq k} \| \nabla P_{j} A_{x} \|_{L_{s}^{2} L_{t}^{4} L_{x}^{2}} \| \psi_{t} \|_{L_{s}^{2} L_{t,x}^{4}}
+
2^{k} \| s^{\frac12} P_{k - 5 \leq \cdot \leq k + 5} (A_{x} \psi_{t}) \|_{L_{s}^{\infty} L_{t,x}^{2}} \| A_{x} \|_{L_{s}^{\infty} L_{x}^{2}}
\\
K_3 &:=
\sum_{j \leq k} 2^{j} \| s^{\frac12} P_{j}(A_{x} \psi_{t}) \|_{L_{s}^{\infty} L_{t,x}^{2}} \| P_{> k - 5} A_{x} \|_{L_{s}^{\infty} L_{x}^{2}}
+
\sum_{j > k + 5} 2^{\frac32 k} \| s^{\frac12} P_{j}(A_{x} \psi_{t}) \|_{L_{s}^{\infty} L_{t,x}^{2}} \| s^{\frac14} P_{j} A_{x} \|_{L_{s}^{\infty} L_{x}^{2}}
\\
K_4 &:=
2^{k} \| P_{k - 5 \leq \cdot \leq k + 5} \psi_{x} \|_{L_{s}^{\infty} L_{t,x}^{4}} \| \psi_{t} \|_{L_{s}^{2} L_{t,x}^{4}} \| \psi_{x} \|_{L_{s}^{2} L_{x}^{\infty}}
+
\sum_{j > k + 5} \| P_{j} \psi_{x} \|_{L_{t,x}^{4}} \| P_j \psi_{t} \|_{L_{s}^{2} L_{t,x}^{4}} \sum_{l \leq k - 5} 2^{l} \| P_{l} \psi_{x} \|_{L_{s}^{\infty} L_{x}^{2}}
\\
K_5 &:=
2^{k} \sum_{k + 5 < j \leq l} \| P_{j} \psi_{x} \|_{L_{t,x}^{4}} \| P_{l} \psi_{x} \|_{L_{s,x}^{2}} \| \psi_{t} \|_{L_{s}^{2} L_{t,x}^{4}}
+
\sum_{l \leq j \leq k} 2^{l} \| P_{l} \psi_{x} \|_{L_{s}^{\infty} L_{x}^{2}} \| P_{j} \psi_{x} \|_{L_{t,x}^{4}} \| \psi_{t} \|_{L_{s}^{2} L_{t,x}^{4}}
\end{split}
\]
Then $\sum_{j=1}^5 K_j \lesssim 2^k \upsilon_k$.
Through using the usual bootstrapping methods we may upgrade this estimate to
\begin{equation}\label{usualboot}
 \| \partial_{x} \int_{0}^{s} e^{(s - s') \Delta} P_{k} U_{t}(s') ds' \|_{L_{t,x}^{2}} \lesssim 2^{k} \upsilon_k (1 + s2^{2k})^{-4}
\end{equation}
Finally,
\[
\begin{split}
\partial_{x} e^{s \Delta} \psi_{t}(0) - i e^{s \Delta} \Delta \psi_{x}(0) 
&= 
-\partial_{x} e^{s \Delta} (A_{l} \psi_{l})(0) + i \partial_{l} e^{s \Delta} (\partial_{x} \psi_{l})(0) - i \Delta e^{s \Delta} \psi_{x}(0)
\\
&=
i \partial_{l} e^{s \Delta} (D_{x} \psi_{l})(0) - i \Delta e^{s \Delta} \psi_{x}(0)  
- \partial_{x} e^{s \Delta} (A_{l} \psi_{x})(0) + \partial_{l} e^{s \Delta} (A_{x} \psi_{l})(0)
\\
&=
-\partial_{l} e^{s \Delta} (A_{l} \psi_{l})(0) - \partial_{x} e^{s \Delta} (A_{l} \psi_{x})(0) + \partial_{l} e^{s \Delta} (A_{x} \psi_{l})(0)
\end{split}
\]
As
\[
\begin{split}
\| \nabla e^{s \Delta} (A_{x} \psi_{x}) \|_{L_{t,x}^{2}} 
\lesssim &\;
2^k \| P_{k - 5 \leq k \leq k + 5} \psi_{x} \|_{L_{t,x}^{4}} \| A_{x} \|_{L_{t,x}^{4}} 
+ \sum_{j \leq k} 2^{j} \| P_{j} \psi_{x} \|_{L_{t,x}^{4}} \| P_{k} A_{x} \|_{L_{t}^{4} L_{x}^{\frac43}}
\\
&
+ \sum_{j > k + 5} 2^{k} \| P_{j} \psi_{x} \|_{L_{t,x}^{4}} \| P_{j} A_{x} \|_{L_{t}^{4} L_{x}^{\frac43}}
\end{split}
\]
with right hand side bounded by $2^k \upsilon_k (1 + s2^{2k})^{-4}$, this completes the proof.
\end{proof}

Finally, we prove \eqref{HF0} for the case $| I | < \infty$.
\begin{proof}[Proof of Theorem \ref{thm:HF0} when $|I| < \infty$]
Without loss of generality assume $|I| = 1$.
Let $\chi \in C^\infty$ be a positive bump function
that is identically 1 on $I$ and supported on an interval of length 2.

While \eqref{HF1}, \eqref{HF2} are not sensitive to time cutoffs,
\eqref{HF3} is. However, by combining \eqref{HF2} and \eqref{HF3}, we have
\begin{equation}
\begin{split}
\| (\partial_{t} - i \Delta) \chi(t) (P_{k} \psi_{x}) \|_{L_{t,x}^{2}} 
&\lesssim 
\upsilon_k 2^{k} (1 + s 2^{2k})^{-4} + \| P_{k} \psi_{x} \|_{L_{t}^{\infty} L_{x}^{2}} \\
&\lesssim \upsilon_k (2^{k} + 1) (1 + s 2^{2k})^{-4}
\end{split}
\end{equation}
which is sufficient for the case $k \geq 0$.

For low frequencies, we build up the $X_k$ bound by bounding $\psi_x$ in each function space
that appears in the definition. As noted, we already have \eqref{HF1}. Combining this with
Sobolev embedding, H\"older in time, and the fact that $k \leq 0$, we conclude
\[
\begin{split}
2^{\frac{k}{2}} \sup_{|j - k| \leq 20} \sup_{\theta \in S^{1}} \sup_{|\lambda| < 2^{k - 4}} \| P_{j, \theta} P_k \psi_{x} \|_{L_{\theta, \lambda}^{\infty, 2}} 
&\lesssim 2^{k} \| P_k \psi_{x} \|_{L_{t,x}^{2}} 
\lesssim \upsilon_k (1 + s2^{2k})^{-4}
\end{split}
\]
By Sobolev embedding, H\"older's inequality, and interpolating \eqref{HF1} and \eqref{DoL4},
we obtain
\[
\begin{split}
2^{\frac{k}{6}} \| P_{j, \theta} P_k \psi_{x}(s) \|_{L_{\theta}^{6,3}(I)} 
&\lesssim 
2^{\frac{k}{3}} \| P_{j, \theta} P_k \psi_{x}(s) \|_{L_{t,x}^{3}(I)} 
\lesssim 
2^{\frac{k}{3}} \| P_{j, \theta} P_k \psi_{x}(s) \|_{L_{t}^{6} L_{x}^{3}(I)} 
\lesssim 
\upsilon_k (1 + s 2^{2k})^{-4}
\end{split}
\]
for $|j - k| \leq 20$, $\theta \in \Sp^1$, and $k \leq 0$.
It remains to prove
\begin{equation} \label{maxest}
\| P_{k} \psi_{x}(s) \|_{L_{x}^{2} L_{t}^{\infty}(I)} \lesssim \upsilon_k (1 + s 2^{2k})^{-4}
\end{equation}
for $k \leq 0$.
To prove \eqref{maxest}, it suffices to prove
\begin{equation} \label{maxestsuff}
\lVert \partial_t P_k \psi_x(s) \rVert_{L^2_x L^1_t} \lesssim \upsilon_k (1 + s 2^{2k})^{-1}
\end{equation}
in view of the energy estimate \eqref{HF1} and the fundamental theorem of calculus.
We do this in several steps.
To begin, we take advantage of the compatibility relation \eqref{txcompexp}
and expand the $\psi_\alpha$ terms using \eqref{Duhamel}. We then control each term
individually.

\textbf{The term $\partial_x \psi_t$}

The control the linear contribution from $\partial_x \psi_t$, we use the fact that
at heat time $s = 0$ we have the relation \eqref{SM-freeze}.
Hence we control
\[
\| e^{s \Delta} P_{k} \partial_{x} \partial_{l} \psi_{l}(0) \|_{L_{t}^{1} L_{x}^{2}} 
\lesssim \upsilon_k 2^{2k} (1 + s 2^{2k})^{-4} 
\lesssim \upsilon_k (1 + s 2^{2k})^{-4}
\]
(as $k \leq 0$). It also holds that
\[
\begin{split}
\| e^{s \Delta} \partial_{x} P_{k} (A_{l}(0) \psi_{l}(0)) \|_{L_{t}^{1} L_{x}^{2}}
\lesssim &\;
2^{k} \sum_{j \leq k} 2^{j} \| P_{j} \psi_{l}(0) \|_{L_{t}^{\infty} L_{x}^{2}} \| A_{x} \|_{L_{t}^{\infty} L_{x}^{2}}
\\
&
+ 2^{k} (1 + s 2^{2k})^{-4} \| P_{k - 5 \leq \cdot \leq k + 5} \psi_{x}(0) \|_{L_{t,x}^{4}} \| A_{x} \|_{L_{t,x}^{4}}
\\
&
+2^{2k} \sum_{j > k + 5} 2^{-j} \| P_{j} \psi_{x}(0) \|_{L_{t}^{\infty} L_{x}^{2}} \| P_{j} \nabla A_{x} \|_{L_{t,x}^{2}}
\end{split}
\]
and the right hand side is bounded by $\upsilon_k (1 + s 2^{2k})^{-4}$.
To control the nonlinear Duhamel term, we use \eqref{usualboot} and the fact that $|I| = 1$ to obtain
\[
\| P_{k} \partial_{x} \int_{0}^{s} e^{(s - s') \Delta} U_{t}(s') ds' \|_{L_{t}^{1} L_{x}^{2}} 
\lesssim \upsilon_k 2^{k} (1 + s 2^{2k})^{-4} \lesssim \upsilon_k (1 + s 2^{2k})^{-4}
\]
Therefore we conclude
\[
 \| \partial_{x} P_{k} \psi_{t} \|_{L_{x}^{2} L_{t}^{1}} \lesssim \upsilon_k (1 + s 2^{2k})^{-4},
 \quad \quad
 k \leq 0
\]

\textbf{The term $A_t \psi_x$}

Here we use \eqref{HF2} to obtain
\[
\begin{split}
\| P_{k} (A_{t} \psi_{x}) \|_{L_{x}^{2} L_{t}^{1}} 
\lesssim &\;
\| A_{t} \|_{L_{t,x}^{2}} \sum_{j \leq k} 2^{j} \| P_{j} \psi_{x} \|_{L_{t}^{\infty} L_{x}^{2}}
+
2^{k} \| P_{k - 5 \leq \cdot \leq k + 5} \psi_{x} \|_{L_{t}^{\infty} L_{x}^{2}} \| A_{t} \|_{L_{t,x}^{2}}
\\
&
+ \sum_{j > k + 5} 2^{\frac{k}{2}} \sup_{|\ell - j | \leq 20} \sup_{\theta \in \Sp^1}
\| P_{\ell, \theta} P_{j} e^{s \Delta} \psi_{x}(0) \|_{L_{\theta}^{6,3}} \| P_{j} A_{t} \|_{L_{t,x}^{2}} 
\\
&
+ \sum_{j > k + 5} 2^{k} \| P_{j} \int_{0}^{s} e^{(s - s') \Delta} U_{x}(s') \|_{L_{t,x}^{2}} \| P_{j} A_{t} \|_{L_{t,x}^{2}}
\end{split}
\]
The right hand side is bounded by $\upsilon_k(1 + s 2^{2k})^{-4}$, as desired.

\textbf{The term $A_x \psi_t$}

We again take advantage of \eqref{txcompexp} and \eqref{Duhamel}.

The nonlinear Duhamel term is bounded by
\[
\begin{split}
\| P_{k} \left[A_{x} \int_{0}^{s} e^{(s - s') \Delta} U_{t}(s')ds^\prime \right] \|_{L_{t}^{1} L_{x}^{2}} 
&
\lesssim 2^{\frac{k}{2}} \| \int_{0}^{s} e^{(s - s') \Delta} U_{t}(s') ds' \|_{L_{t,x}^{2}} 
\sum_{j \leq k + 5} 2^{\frac{j}{2}} \| P_{j} A_{x} \|_{L_{t,x}^{4}}
\\
&\quad
+ 2^{k} \sum_{j > k + 5} \| P_{j} A_{x} \|_{L_{t,x}^{2}} \| P_{j} \int_{0}^{s} e^{(s - s') \Delta} U_{t}(s') ds' \|_{L_{t,x}^{2}}
\end{split}
\]
and the right hand side is bounded by $\upsilon_k (1 + s 2^{2k})^{-4}$.

This leaves the terms coming from the linear evolution.
The first of these we control with
\[
\begin{split}
 \| P_{k} (A_{x} e^{s \Delta} \partial_{l} \psi_{l}(0)) \|_{L_{x}^{2} L_{t}^{1}} 
 &\lesssim \;
 \sum_{j \leq k - 5} \|  A_{x} \|_{L_{t,x}^{4}} \| P_{j} e^{s \Delta} (\partial_{l} \psi_{l}(0)) \|_{L_{t,x}^{4}} 
 \\
 &
\quad + \| P_{k - 5 \leq \cdot \leq k + 5} e^{s \Delta} \partial_{l} \psi_{l}(0) \|_{L_{t,x}^{4}} \| A_{x} \|_{L_{t,x}^{4}}
 \\
 &
\quad +2^{\frac{k}{2}} \sum_{j \geq k + 5} \| P_{j} A_{x} \|_{L_{t,x}^{2}} 
 \sup_{|p - j| \leq 20} \sup_{\theta \in \Sp^1} \| e^{s \Delta} P_{p, \theta} P_{j} \partial_{l} \psi_{l}(0) \|_{L_{\theta}^{6,3}}
\end{split}
\]
whose right hand side is bounded by $\upsilon_k (1 + s 2^{2k})^{-4}$.
For the second, we have
\[
\begin{split}
&\| P_{k} (e^{s \Delta} A_{l}(0) \psi_{l}(0)) A_{x} \|_{L_{t}^{1} L_{x}^{2}} 
\lesssim
\\
&\quad \quad
(1 + s 2^{2k})^{-4} \| \psi_{x}(0) \|_{L_{t,x}^{4}} \| A_{x}(0) \|_{L_{t,x}^{4}} \sum_{j \leq k - 5} 2^{\frac{j}{2}} \| P_{j} A_{x}(s) \|_{L_{t,x}^{4}}
\\
&\quad \quad
+\| P_{k - 5 \leq \cdot \leq k + 5} A_{x} \|_{L_{t,x}^{4}} \| A_{l}(0) \| _{L_{t,x}^{4}} \| \psi_{l}(0) \|_{L_{t,x}^{4}}
\\
&\quad \quad
+\sum_{j > k + 5} 2^{k}  \| P_{j} A_{x}(s) \|_{L_{t,x}^{2}} 
\left[ 
(1 + s2^{2k})^{-4} \| P_{j} \psi_{x}(0) \|_{L_{t,x}^{4}} \| A_{x}(0) \|_{L_{t,x}^{4}} + \right.
\\
&\quad \quad \quad \quad \quad \quad \quad \quad \quad \quad \quad \quad \quad \quad
\left. \| \psi_{x}(0) \|_{L_{t,x}^{4}} \| P_{j} A_{x}(0) \|_{L_{t,x}^{4}} \right]
\end{split}
\]
and the right hand side is bounded by $\upsilon_k (1 + s 2^{2k})^{-4}$.
\end{proof}

\section{Closing the argument} \label{sec:closer}

We present two different arguments for handling the part of the nonlinearity $\cN_m$
not covered by Lemma \ref{lem:Npert}. The first argument relies on a time subdivision,
and the second on a discrete Gronwall-type approach.

\subsection{Argument I}

It suffices to control either $\nabla \cdot (A_x \psi_x)$ or $A_x \nabla \psi_x$, as they are equivalent up to a
$(\nabla \cdot A_x) \psi_x$ term, which was controlled in \S \ref{sec:mainresult}.

Partition $I$ into finitely many pieces such that
\begin{equation} \label{smallintervals}
 \sum_{k} \sum_{l \leq k} \int_{0}^{\infty} \| (P_{k} \psi_{x}(0))(P_{l} \psi_{x}(s)) \|_{L_{t,x}^{2}(I_{j})}^{2} 2^{(k - l)} 2^{2l} ds \leq \epsilon
\end{equation}
on each subinterval $I_{j}$. 

If
\[
 \sum_{k} \| P_{k} \psi_{x} \|_{X_{k}(I_{j})}^{2} < \infty
\]
then there exists a frequency envelope $a_{k}$ such that
\[
  \| P_{k} \psi_{x} \|_{X_{k}(I_{j})} \leq a_{k}
\]
for all $k$ and
\[
 \sum_{k} a_{k}^{2} \lesssim \sum_{k} \| P_{k} \psi_{x} \|_{X_{k}(I_{j})}^{2}
\]
Therefore Theorem \ref{thm:HF0} implies
\[
 \sum_{k} (\sup_{s} \| P_{k} \psi_{x}(s) \|_{X_{k}(I_{j})})^{2} \lesssim \sum_{k} \| P_{k} \psi_{x} \|_{X_{k}(I_{j})}^{2}
\]

We look at the term
\[
 \nabla \cdot (A_{x} \psi_{x}) = -\nabla \cdot (\psi_{x} \int_{0}^{\infty} \Im(\bar{\psi}_{x} D_{\ell} \psi_{\ell})(s) ds)
\]
and decompose it into two main pieces.

First, we consider
\begin{equation} \label{Adpsi-case1}
\sum_{k} \| P_{k}((\nabla \psi_{x}) \cdot \int_{0}^{\infty} \Im(\bar{\psi}_{x} \partial_{\ell} \psi_{\ell})(s) ds) \|_{Y_{k}(I_{j})}^{2}
\lesssim
K_1 + K_2 + K_3 + K_4
\end{equation}
where here
\[
\begin{split}
K_1 &:=
\sum_{k} 2^{k} (\sum_{l_{1} \leq l_{2} \leq k - 10} a_{l_{1}} 2^{l_{2}} 2^{\frac{l_{1}}{2}} \cdot \int_{0}^{\infty}  \| (P_{k - 5 \leq \cdot \leq k + 5} \psi_{x}(0)) (P_{l_{2}} \psi_{x}(s)) \|_{L_{t,x}^{2}(I_{j})}  ds)^{2}
\\
K_2 &:=
\sum_k (2^{2k} \sum_{j > k - 10} \int_0^\infty \lVert P_j \psi_x(s) \rVert_{L^4_{t,x}} ds )^2 a_k^2 \epsilon^2
\\
K_3 &:= \sum_k 2^{2k} ( \sum_{j > k + 10} a_j \sum_{j_1 > j - 10} 2^{j_1} \lVert P_{j_1} \psi_x(s) \rVert_{L^4_{t,x}} ds)^2
\epsilon^2
\\
K_4 &:= \sum_k ( \sum_{j < k - 10} \lVert (P_j \psi_x) (\nabla P_{k - 10 \leq \cdot \leq k+10} A_x) \rVert_{Y_k(I_j)})^2
\end{split}
\]
By \eqref{bilinear},
\[
K_4 \lesssim \epsilon^4 \sum_k (\sum_{j < k -10} a_j 2^{\frac{j-k}{2}})^2 \lesssim \epsilon^4 (\sum_k a_k^2)
\]
Combining \eqref{smallintervals} with the Cauchy-Schwarz inequality,
\[
\sum_{k} \| P_{k}((\nabla \psi_{x}) \cdot \int_{0}^{\infty} \Im(\bar{\psi}_{x} \partial_{\ell} \psi_{\ell})(s) ds) \|_{Y_{k}(I_{j})}^{2}
 \lesssim \epsilon (\sum_{l} a_{l}^{2})
\]

Now consider the term
\[
 \sum_{k} \| P_{k}((\nabla \psi_{x}) \cdot \int_{0}^{\infty} \Re(\bar{\psi}_{x} A_{\ell} \psi_{\ell})(s) ds) \|_{Y_{k}(I_{j})}^{2}
\]
As
\[
\begin{split}
 \| P_{l} (A_{x} \psi_{x}) \|_{L_{t,x}^{4}} 
\lesssim&\;
2^{l} \| P_{l - 5 \leq \cdot \leq l + 5} \psi_{x} \|_{L_{t,x}^{4}} \| A_{x} \|_{L_{t}^{\infty} L_{x}^{2}}
+ 2^{\frac{l}{2}} \sum_{l_{1} \leq l - 5} 2^{\frac{l_{1}}{2}} \| P_{l_{1}} \psi_{x} \|_{L_{t,x}^{4}} \| P_{> l - 5} A_{x} \|_{L_{t}^{\infty} L_{x}^{2}}
\\
&
+ 2^{\frac{3l}{2}} \sum_{j > l + 5} \| P_{j} \psi_{x} \|_{L_{t,x}^{4}} \| P_{> j} A_{x} \|_{L_{t}^{\infty} L_{x}^{\frac43}}
\end{split}
\]
with right hand side bounded by $\upsilon_l(1 + s 2^{2l})^{-4}$, it therefore follows that
terms involving $P_{> k - 10} (A_{x} \psi_{x})(s)$ or $P_{> k - 10} \psi_{x}(s)$ can be analyzed in identical fashion to 
those appearing in the last two lines of \eqref{Adpsi-case1}.
This leaves us with a term of the form
\[
 \sum_{k} \| P_{k}((\nabla \psi_{x}) \cdot \sum_{l_{1} \leq l_{2} \leq l - 10} \int_{0}^{\infty} \Re(P_{l_{1}}(\bar{\psi}_{x}) P_{l_{2}}(A_{\ell} \psi_{\ell}))(s) ds) \|_{Y_{k}(I_{j})}^{2}
\]
Invoking \eqref{bilinear} and by choosing $C_{0}(E_{0})$ to be a very large, fixed constant,
we bound this term by
\[
\begin{split}
&\sum_{k} 2^{k}2^{\frac{C_{0}}{2}} (\sum_{l_{1} \leq l_{2} \leq k - 10} a_{l_{1}} 2^{l_{2}} 2^{\frac{l_{1}}{2}} \cdot \int_{0}^{\infty}  \| (P_{k - 5 \leq \cdot \leq k + 5} \psi_{x}(0)) (P_{l_{2}} \psi_{x}(s)) \|_{L_{t,x}^{2}(I_{j})}  ds)^{2}
\\
&+ \sum_{k} 2^{k} ( \sum_{l_{1} \leq l_{2} \leq 2^{k - 10} : l_{3} \geq l_{2} + C_{0} } (\sup_{s} \| (P_{k - 5 \leq \cdot \leq k + 5} \psi_{x}(0)) (P_{l_{1}} \psi_{x}(s)) \|_{L_{t,x}^{2}}) 2^{\frac{l_{3}}{2}} a_{l_{3}} \| P_{l_{3}} A_{x} \|_{L_{s}^{1} L_{x}^{\infty}})^{2}
\end{split}
\]
which is controlled by $\epsilon 2^{\frac{C_{0}}{2}} (\sum_{l} a_{l}^{2}) + 2^{-\frac{C_{0}}{2}} (\sum_{l} a_{l}^{2})$.

Choosing $C_{0}(E_{0})$ sufficiently large, $\epsilon(C_{0})$ sufficiently small, then we have, 
since $\sum_{l} b_{l}^{2} \lesssim \sum_{l} \| P_{l} \psi_{x}(0) \|_{X_{l}(I_{j})}^{2}$,
\[
 \sum_{k} \| P_{k} \psi_{x}(0) \|_{X_{k}(I_{j})}^{2} \lesssim_{E_{0}, C_{0}} 1
\]
Plugging the finiteness of $\sum_{l} b_{l}^{2}$ back into the above yields
\[
 \| P_{k} \psi_{x}(0) \|_{X_{k}(I_{j})}  \lesssim \upsilon_k
\]

\subsection{Argument II}

Our goal is to control $A_\ell \partial_\ell \psi_m$.
To do so, we begin by decomposing it according to the usual
Littlewood-Paley trichotomy:
\begin{equation} \label{Adpsi-trich}
P_k \left(A_\ell \partial_\ell \psi_m\right)
=
P_k [ (
\sum_{\substack{k_1\leq k - 5 \\ |k_2 - k| \leq 4}} +
\sum_{\substack{k_2\leq k - 5 \\ |k_1 - k| \leq 4}} +
\sum_{\substack{k_1, k_2 \geq k - 5 \\ |k_1 - k_2| \leq 8}} 
)
P_{k_1} A_\ell \partial_\ell P_{k_2} \psi_m ]
\end{equation}
\begin{rem}
In this section we slightly abuse asymptotic notation, as we endow it with a meaning different from
its usual one when applied to indices indicating frequency projections.
For instance, if $k_1$, $k_2$ are indices associated to the Littlewood-Paley projections $P_{k_1}$, $P_{k_2}$, 
then the expression $k_1 \lesssim k_2$ is a shorthand for $2^{k_1} \lesssim 2^{k_2}$, etc. 
In other circumstances, the asymptotic notation retains its usual meaning.
\end{rem}
When $k_2 \lesssim k_1 \sim k$ as in the second sum of \eqref{Adpsi-trich}, we
treat the derivative on $\psi_m$ as $2^{k_2}$, transfer it to $A_\ell$,
and then use the $L^2_{t,x}$ bound on $\partial_x A_x$.
\begin{lem} \label{lem:HL}
It holds that
\[
\lVert
\sum_{\substack{k_2\leq k - 5 \\ |k_1 - k| \leq 4}}
P_{k_1} A_\ell \partial_\ell P_{k_2} \psi_m
\rVert_{Y_k(I)}
\lesssim \varepsilon^2 b_k
\]
\end{lem}
\begin{proof}
By Lemmas \ref{lem:YkL2} and \ref{lem:L2bounds} 
and the frequency envelope property, we have
\[
\begin{split}
\lVert
\sum_{\substack{k_2\leq k - 5 \\ |k_1 - k| \leq 4}}
P_{k_1} A_\ell \partial_\ell P_{k_2} \psi_m
\rVert_{Y_k(I)}
&\lesssim
\sum_{\substack{k_2\leq k - 5 \\ |k_1 - k| \leq 4}}
2^{-|k_2 - k_1|} \lVert 2^{k_1} P_{k_1} A_x \rVert_{L^2_{t,x}} \lVert P_{k_2} \psi_m \rVert_{X_{k_2}(I)}
\\
&\lesssim
\sum_{\substack{k_2\leq k - 5 \\ |k_1 - k| \leq 4}}
2^{-|k_2 - k_1|}
\varepsilon^2 b_{k_2}
\lesssim
\varepsilon^2 b_k
\end{split}
\]
\end{proof}
For the high frequencies we lose summability.
This can be overcome by expanding $A$ using the representation \eqref{A}
and refining our analysis.
Hence we consider
\begin{equation} \label{Adpsi}
P_k (A_\ell \partial_\ell \psi_m) =
-P_k[
\sum_{k_1, k_2, k_3}
\int_0^\infty P_{k_1} \psi_\ell(s^\prime) P_{k_2} \psi_s(s^\prime) ds^\prime \partial_\ell P_{k_3} \psi_m
]
\end{equation}
In \eqref{Adpsi},
we expand $\psi_s$ using \eqref{SM-freeze}, so that 
$\psi_s(s^\prime) = \partial_j \psi_j(s^\prime) + (i A_j \psi_j)(s^\prime)$,
and then treat the $\partial_j \psi_j(s^\prime)$ and $(i A_j \psi_j)(s^\prime)$ terms separately.
\begin{lem} \label{lem:HH}
It holds that
\[
\lVert P_k [
\sum_{\substack{k_1, k_2 \geq k - 5 \\ |k_1 - k_2| \leq 8}} 
P_{k_1} A_\ell \partial_\ell P_{k_2} \psi_m ]
\rVert_{Y_k(I)} \lesssim \upsilon_k^2 b_k
\]
\end{lem}
\begin{proof}
For the summation under consideration,
in \eqref{Adpsi} we are restricted to the range $k_3 \gtrsim k$ and $\max\{k_1, k_2\} \gtrsim k_3$.

First we treat only the $\partial_x \psi_x$ portion of the $\psi_s$ term. Hence we need to control
\[
K_1 := P_k[ \sum_{\substack{k_3 \gtrsim k \\ \max\{k_1, k_2\} \gtrsim k_3}}
\int_0^\infty
P_{k_1} \psi_x(s^\prime) \partial_x P_{k_2} \psi_x (s^\prime) ds^\prime \partial_x P_{k_3} \psi_x(0) ]
\]
Suppose $k_1 \geq k_2$. Then
we apply \eqref{bilinear}, pulling out $P_{k_1} \psi_x$ in $X_{k_1}(I)$ 
and picking up a gain of at least $2^{- \frac{|k - k_1|}{6}}$.
To the remaining term we apply \eqref{bls}.
Then
\[
\lVert K_1|_{k_1 \geq k_2} \rVert_{Y_k(I)}
\lesssim
\sum_{\substack{k_1 \gtrsim k_3 \gtrsim k \\ k_1 \geq k_2}}
b_{k_1} \upsilon_{k_2} \upsilon_{k_3} 
2^{-\frac{|k - k_1|}{6}} 2^{-\frac{|k_3 - k_2|}{2}} 2^{k_2 - k_1} 2^{k_3 - k_1}
\]
Using the frequency envelope summation properties, we can sum the geometric series.
If, on the other hand, $k_2 > k_1$, then we pull out
$P_{k_3} \psi_x$ in $X_{k_3}(I)$ using \eqref{bilinear} instead.
We obtain
\[
\lVert K_1|_{k_2 > k_1} \rVert_{Y_k(I)}
\lesssim
\sum_{\substack{k_2 \gtrsim k_3 \gtrsim k \\ k_2 > k_1}}
b_{k_3} \upsilon_{k_1} \upsilon_{k_2}
2^{-\frac{|k - k_3|}{6}} 2^{-\frac{|k_1 - k_2|}{2}} 2^{k_3 - k_2}
\]
Combining the two cases, we conclude
\begin{equation} \label{geosum}
\lVert K_1 \rVert_{Y_k(I)} \lesssim \upsilon_k^2 b_k
\end{equation}
\begin{rem}
When $k_2 > k_1$, one could pull out $P_{k_2} \psi_x$ in $X_{k_2}(I)$ rather
than $P_{k_3} \psi_x$ in $X_{k_3}(I)$; this strategy, however, poses additional challenges
when it comes to the handling the analogous $P_{j_1} A_x P_{j_2} \psi_x$ term (see below),
as the frequency localizations in the definition of the $X_k$ norm prevent us from
directly pulling out $P_{j_1} A_x$ in $L^\infty_{t,x}$.
\end{rem}
Next we consider
\[
K_2 := P_k[ \sum_{\substack{k_3 \gtrsim k \\ \max\{k_1, k_2\} \gtrsim k_3}}
\int_0^\infty
P_{k_1} \psi_x(s^\prime) 
P_{k_2}[ \sum_{j_1, j_2} (P_{j_1} A_x P_{j_2} \psi_x)(s^\prime) ] 
ds^\prime \partial_x P_{k_3} \psi_x(0) ]
\]
If $k_1 \geq k_2$, then,
as above, we apply \eqref{bilinear} and pull out $P_{k_1} \psi_x$ in $X_{k_1}(I)$.
To the rest we apply \eqref{bls}, but pull out $P_{j_1} A_x$ in $L^\infty$, i.e.,
\[
\begin{split}
\lVert
P_{k_2}[(P_{j_1} A_x P_{j_2} \psi_x)(s^\prime) ] 
\partial_x P_{k_3} \psi_x(0)
\rVert_{L^2_{t,x}}
&\lesssim
\lVert P_{j_1} A_x(s^\prime) \rVert_{L^\infty_{t,x}} 
\lVert P_{j_2} \psi_x(s^\prime) \cdot P_{k_3} \partial_x \psi_x(s^\prime) \rVert_{L^2_{t,x}}
\\
&\lesssim
\lVert P_{j_1} A_x(s^\prime) \rVert_{L^\infty_{t,x}} 
2^{k_3} 2^{- \frac{|k_3 - j_2|}{2}} \upsilon_{j_2} \upsilon_{k_3} (1 + s 2^{2j_2})^{-1} (1 + s 2^{2k_3})^{-1}
\end{split}
\]
We use \eqref{A-energy} to obtain
\[
\lVert P_{j_1} A_x(s^\prime) \rVert_{L^\infty_{t,x}} \lesssim_{E_0} 2^{j_1} (1 + s 2^{2j_1})^{-1}
\]
When $j_1, j_2 \lesssim k_2$, the contribution in $K_2$ is weaker than the corresponding
contribution from $\partial_x P_{k_2} \psi_x$ in $K_1$.
The summation of $j_1$ up to $k_2$ is taken directly and
the summation of $j_2$ up to $k_2$ is achieved using the decay from the application of \eqref{bls}:
\[
\lVert K_2|_{\substack{k_1 \geq k_2 \\ j_1, j_2 \leq k_2}} \rVert_{Y_k(I)}
\lesssim
\sum_{\substack{k_1 \gtrsim k_3 \gtrsim k \\ k_1 \geq k_2}}
\sum_{j_1, j_2 \leq k_2}
b_{k_1} \upsilon_{j_2} \upsilon_{k_3}
2^{- \frac{|k - k_1|}{6}} 2^{- \frac{|k_3 - j_2|}{2}} 2^{j_1 + k_3} 2^{- 2k_1}
\]
When $j_1 \sim j_2 \gtrsim k_2$, we 
have extra decay from the heat flow and from \eqref{bls}, which enable us to sum:
\[
\lVert K_2|_{\substack{k_1 \geq k_2 \\ j_1 \sim j_2 \geq k_2}} \rVert_{Y_k(I)}
\lesssim
\sum_{\substack{k_1 \gtrsim k_3 \gtrsim k \\ k_1 \geq k_2}}
\sum_{j_1\sim j_2 \geq k_2}
b_{k_1} \upsilon_{j_2} \upsilon_{k_3}
2^{- \frac{|k - k_1|}{6}} 2^{- \frac{|k_3 - j_2|}{2}} 2^{j_1 + k_3} 2^{- 2\max\{j_1, k_1\}}
\]
If $k_2 > k_1$, then we pull out $P_{k_3} \psi_x$ in $X_{k_3}(I)$ instead and control
the remaining terms in $L^2$ as above.
Combining the cases, we conclude
\[
\lVert K_2 \rVert_{Y_k(I)} \lesssim \upsilon_k^2 b_k
\]
\end{proof}
This leaves in \eqref{Adpsi-trich} only the first term
\[
\sum_{\substack{k_1\leq k - 5 \\ |k_2 - k| \leq 4}} P_{k_1} A_\ell \partial_\ell P_{k_2} \psi_m
\]
to control.
The next lemma shows that part of this term enjoys bounds similar to those
in the previous lemma.
\begin{lem} \label{lem:LH-1}
It holds that
\[
\lVert P_k [
\sum_{\substack{ |k_1 - k_2| \leq 8, k_1 \geq k - 4 \\ |k_3 - k| \leq 4}}
\int_0^\infty
P_{k_1} \psi_\ell(s^\prime)
P_{k_2} \psi_s(s^\prime) ds^\prime
\partial_\ell P_{k_3} \psi_m(0)
]
\rVert_{Y_k(I)} \lesssim \upsilon_k^2 b_k
\]
\end{lem}
\begin{proof}
The proof is analogous to that of Lemma \ref{lem:HH}.
We treat
\[
K_1 := 
P_k[ 
\sum_{\substack{ |k_1 - k_2| \leq 8, k_1 \geq k - 4 \\ |k_3 - k| \leq 4}}
\int_0^\infty P_{k_1} \psi_x(s^\prime) \partial_x P_{k_2} \psi_x(s^\prime) ds^\prime \partial_x P_{k_3} \psi_x
]
\]
and
\[
K_2 := P_k[ 
\sum_{\substack{ |k_1 - k_2| \leq 8, k_1 \geq k - 4 \\ |k_3 - k| \leq 4}}
\int_0^\infty
P_{k_1} \psi_x(s^\prime) 
P_{k_2}[ \sum_{j_1, j_2} (P_{j_1} A_x P_{j_2} \psi_x)(s^\prime) ] 
ds^\prime \partial_x P_{k_3} \psi_x(0) ]
\]
separately, but in both cases
apply \eqref{bilinear} to pull out $P_{k_1} \psi_x$ in $X_{k_1}$.
We can always apply estimate \eqref{bilinear} in this way because
here $k_1 \gtrsim k_2$.
Next we use \eqref{bls} and, in the case of $K_2$, bound $P_{j_1} A_x$
in $L^\infty_{t,x}$. We are left with summable geometric series in both cases.
\end{proof}
Combining Lemmas \ref{lem:Npert}, \ref{lem:HL}, \ref{lem:HH}, and \ref{lem:LH-1}, we conclude
\begin{lem} \label{lem:Npert2}
It holds that
\[
\lVert
P_k[
\cN_m - 2i
\sum_{\substack{k_1, k_2 \ll k \\ |k_3 - k| \leq 4}}
\int_0^\infty \Im(\overline{P_{k_1} \psi_j} P_{k_2} \psi_s)(s^\prime) ds^\prime P_{k_3} \partial_j \psi_m(0)
]
\rVert_{Y_k(I)}
\lesssim (\varepsilon^2 + \upsilon_k^2) b_k
\]
\end{lem}

Now we pick up with the term of $\cN_m$ left unaddressed by Lemma \ref{lem:Npert2}.
The term does not enjoy bounds like those above, but does still obey bounds good
enough for closing a bootstrap argument.
This approach to closing the argument is the same as the one used in \cite{Sm10, Sm11}.
\begin{lem} \label{lem:Nnonpert}
It holds that
\[
\lVert
P_k \sum_{\substack{k_1, k_2 \ll k \\ |k_3 - k| \leq 4}}
\int_0^\infty \Im(\overline{P_{k_1} \psi_j} P_{k_2} \psi_s)(s^\prime) ds^\prime P_{k_3} \partial_j \psi_m(0)
\rVert_{Y_k(I)}
\lesssim \upsilon_k \sum_{j \leq k - C_1} b_j \upsilon_j
\]
\end{lem}
\begin{proof}
As in the proofs of Lemmas \ref{lem:HH} and \ref{lem:LH-1}, we use \eqref{bilinear}
to pull out $P_{k_1} \psi_x$ in $X_{k_1}(I)$. Here, however, we need the full one-half
power of decay afforded us by \eqref{bilinear}. As previously, to the remaining term we apply
\eqref{bls} to control it in $L^2$. When we consider the $\partial_x \psi_x$ portion of $\psi_s$, we have
the bound
\[
\sum_{\substack{k_1, k_2 \ll k \\ |k_3 - k| \leq 4}} b_{k_1} \upsilon_{k_2} \upsilon_{k_3}
2^{- \frac{|k_1 - k|}{2}} 2^{- \frac{|k_3 - k_2|}{2}} 2^{k_2 + k_3} 2^{- 2\max\{k_1, k_2\}}
\]
In this expression $k_1$ and $k_2$ play symmetric roles and so we do not need to consider
the cases $k_1 \geq k_2$ and $k_2 > k_1$ separately. The decay along the heat flow
allows us to sum to the diagonal, but not all the way up to $k$: there is no extra source
of decay along the diagonal $k_1 = k_2$,
which is why the best bound we can achieve for this term
is $\upsilon_k \sum_{j \leq k - C_1} b_j \upsilon_j$.

As for the $P_{j_1} A_x(s^\prime)$ term appearing in the expansion
of $\psi_s$, we place it in $L^\infty_{t,x}$ as done previously, and recover summation to the diagonal.
\end{proof}

Combining Proposition \ref{prop:MainLinearEstimate} and Lemmas \ref{lem:Npert2} and \ref{lem:Nnonpert},
we obtain
\[
\lVert P_k \psi_m \rVert_{X_k(I)}
\lesssim
\lVert P_k \psi_m(t_0) \rVert_{L^2_x} + (\varepsilon^2 + \upsilon_k^2) b_k + 
\upsilon_k \sum_{j \leq k - C_1} b_j \upsilon_j
\]
The initial data is included in the $\upsilon_k$ envelope; moreover, by using the frequency
envelope property, we may absorb the $\upsilon_k^2 b_k$ term in the summation.
Therefore
\[
b_k \lesssim \upsilon_k +
\upsilon_k \sum_{j \leq k - C_1} b_j \upsilon_j
\]
Squaring and applying Cauchy-Schwarz yields
\begin{equation}
b_k^2 \lesssim  (1 +  \sum_{j \leq k - C_1} b_j^2 ) \upsilon_k^2
\label{SCS}
\end{equation}
Setting
\begin{equation*}
B_k := 1 + \sum_{j < k} b_j^2
\end{equation*}
in (\ref{SCS}) leads to
\begin{equation*}
B_{k+1} \leq B_k (1 + C_2 \upsilon_k^2)
\end{equation*}
with $C_2 > 0$ independent of $k$.
Therefore
\begin{equation*}
B_{k + m} 
\leq 
B_k \prod_{\ell = 1}^m (1 + C_2 \upsilon_{k + \ell}^2)
\leq
B_k \exp(C_2 \sum_{\ell = 1}^m \upsilon_{k + \ell}^2)
\lesssim B_k
\end{equation*}
Since $B_k \to 1$ as $k \to -\infty$, we conclude
\begin{equation*}
B_k \lesssim 1
\end{equation*}
uniformly in $k$, so that, in particular,
\begin{equation}
\sum_{j \in \Z} b_j^2
\lesssim 1
\label{b l2 bounded}
\end{equation}
which, joined with (\ref{SCS}), implies
\[
b_k \lesssim \upsilon_k
\]
as desired.

\emph{Acknowledgments.} 
The second author thanks Ioan Bejenaru for helpful conversations and encouragement in connection with
pursuing this project.

\bibliography{Smbib}
\bibliographystyle{amsplain}

\end{document}